\documentclass[preprint,authoryear]{elsarticle}
\usepackage{amsmath}
\usepackage{amstext}
\usepackage{amsfonts}
\usepackage{amssymb}

\usepackage{graphicx}
\usepackage{bbm}

\usepackage{tikz}
\usepackage{pgfplots}
\usepackage[T2A,T1]{fontenc}

\newcommand{\am}{\text{argmin}}

\newcommand{\un}{\mathbbm{1}}

\newcommand{\y}{{\bf y}}

\newcommand{\lip}{\mathcal L \textit{ip}}

\DeclareMathOperator{\tr}{Tr}

\DeclareMathOperator{\diag}{Diag}
\DeclareMathOperator{\dimm}{dim}
\DeclareMathOperator{\vect}{Vect}
\def\restriction#1#2{\mathchoice
              {\setbox1\hbox{${\displaystyle #1}_{\scriptstyle #2}$}
              \restrictionaux{#1}{#2}}
              {\setbox1\hbox{${\textstyle #1}_{\scriptstyle #2}$}
              \restrictionaux{#1}{#2}}
              {\setbox1\hbox{${\scriptstyle #1}_{\scriptscriptstyle #2}$}
              \restrictionaux{#1}{#2}}
              {\setbox1\hbox{${\scriptscriptstyle #1}_{\scriptscriptstyle #2}$}
              \restrictionaux{#1}{#2}}}
\def\restrictionaux#1#2{{#1\,\smash{\vrule height .8\ht1 depth .85\dp1}}_{\,#2}} 

\newcounter{ctheorem}
\newtheorem{theorem}[ctheorem]{Theorem}
\newtheorem{assumption}{Assumption}

\usepackage{color}
\definecolor{myblue}{RGB}{0,70,110}

\newproof{proof}{Proof}

\newcounter{cexample}
\newtheorem{example}[cexample]{Example}
\newcounter{cproposition}
\newtheorem{proposition}[cproposition]{Proposition}

\newcounter{ccorollary}
\newtheorem{corollary}[ccorollary]{Corollary}
\newcounter{cconjecture}
\newtheorem{conjecture}[cconjecture]{Claim}
\newcounter{clemma}
\newtheorem{lemma}[clemma]{Lemma}
\newcounter{cdefinition}
\newtheorem{definition}[cdefinition]{Definition}
\newtheorem{defpro}[cdefinition]{Definition/Proposition}
\newcounter{cremark}
\newtheorem{remark}[cremark]{Remark}

\pgfplotsset{compat=1.11}

\begin{document}

\begin{frontmatter}

\title{A Concentration of Measure Framework to study convex problems and other implicit formulation problems in machine learning}




\author[gipsa]{Cosme Louart}
\ead{cosmelouart@gmail.com}




\begin{abstract}
This paper provides a framework to show the concentration of solutions $Y^*$ to convex minimizing problem where the objective function $\phi(X)(Y)$ depends on some random vector $X$ satisfying concentration of measure hypotheses. More precisely, the convex problem translates into a contractive fixed point equation that ensure the transmission of the concentration from $X$ to $Y^*$. This result is of central interest to characterize many machine learning algorithms which are defined through implicit equations (e.g., logistic regression, lasso, boosting, etc.). Based on our framework, we provide precise estimations for the first moments of the solution $Y^*$, when $X= (x_1,\ldots, x_n)$ is a data matrix of independent columns and $\phi(X)(y)$ writes as a sum $\frac{1}{n}\sum_{i=1}^n h_i(x_i^TY)$. That allows to describe the behavior and performance (e.g., generalization error) of a wide variety of machine learning classifiers.

\end{abstract}  


\end{frontmatter}
\tableofcontents
\section*{Introduction}
In modern statistical learning approaches, many optimization problems can be formulated in terms of convex minimization problem or fixed point equation (soft-max classification, M-estimators, low density classifiers, soft snm, logistic regression, optimal transport, adaboost, stochastic gradient descent...). As such, a theoretical characterization of the behavior of such objects becomes crucial in order to assess the performances of the underlying algorithms. 
This paper has two main objectives: first it provides qualitative inferences concerning the statistical concentration of the parameter vector learned during the resolution of the problem and second, some quantitative inferences on a restricted range of regularized empirical risk minimization problems (see \cite{mai2019high}) that write:
\begin{align*}
   \text{Minimize}: \frac{1}{n} \sum_{i=1}^n h_i(x_i^TY) + \lambda \|Y\|^2.
 \end{align*} 
 In this article we stick to $\ell^2$-norm regularization although other kind of regularizations can be found in the literature \cite{celentano2022fundamental}. We however believe in possible generalization of our approach as it was done very recently for the Lasso problem in \cite{tiomoko2022deciphering}.
The class of so called ``ridge regularized robust regression estimators'' (\cite{bean2013optimal,ELK13b}) contains in particular softmax classifiers that were already studied in \cite{mai2019large,SED21}.
There are been continuous work made around this range of problems with different kind of techniques, approximate message passing in \cite{sur2019modern,donoho2016high}, convex Gaussian minmax Theorem in \cite{thrampoulidis2020theoretical,deng2022model}, Meanfield Theory in \cite{mignacco2020dynamical} or Replica Method in \cite{mignacco2020role,saglietti2022solvable}. As a naive comparison to all these works we wish to emphasize the simplicity of the fixed point equation providing the first statistics of $Y$ given here in Claim~\ref{conj:fixed_point_approximation}. 
But the more important contribution is probably a rigorous justification of the leave-one-out technique introduced by El Karoui, they rely on new random matrix theory findings \cite{LOU21HV} concerning the study of matrices $\frac{1}{n}XDX^T$ when $D$ is random and depends on $X$.

Our probabilistic approach, following the one of \cite{ELK13b}, is inspired from Random matrix theory and formulate its hypotheses thanks to the Concentration of Measure Theory (CMT) that demonstrates an interesting flexibility allowing $(i)$ to characterize realistic setting where, in particular, the hypothesis of independent entries is relaxed $(ii)$ to provide rich concentration inequalities with precise convergence bounds (this second aspect was not explored in \cite{ELK13b}). 
To present the simpler picture possible, we will admit in this introduction that what we call for the moment "concentrated vectors" are transformation $F(Z)$ of a Gaussian vector $Z \sim \mathcal N(0,I_d)$ for a given $1$-Lipschitz (for the euclidean norm) mapping $F : \mathbb R^d \to \mathbb R^p$.
A central result of CMT (see for instance \cite[Corollary 2.6]{LED05}) states that 
for any $\lambda$-Lipschitz mapping $f: \mathbb R^d \to \mathbb R$ (where $\mathbb R^d$ and $\mathbb R$ are respectively endowed with the euclidean norm $\|\cdot\|$ and with the absolute value $|\cdot|$):
\begin{align}\label{eq:first_concentration_inequality}
  \forall t >0 : \mathbb P \left(\left\vert f(Z) - \mathbb E[f(Z)]\right\vert \geq t\right) \leq C e^{-(t/\sigma\lambda)^2},
\end{align}
where $C=2$ and $\sigma = \sqrt 2$ (they do not depend on the dimensions $d$ !).

To give a partial but somehow convincing justification for the adoption of this assumption as a basis to machine learning theoretical inferences, we resort to the example of Generative adversarial neural network (GAN) that are able to construct realistic image as Lipshitz transformation of a Gaussian noise, demonstrating this way that some distribution of concentrated  random vectors are very close to the data encountered in machine learning problems (\cite{SED20}).

\section{Main results}
\subsection{Theoretical results}
The paper provides several theoretical results of concentration of solution of random fixed point equation of progressive difficulty, we present below one of the more general and efficient result, Corollary~\ref{cor:concentration_lipschitz_solution_point_fixe2}. Given two normed vector space $(E, \|\cdot\|_E)$ and $(F, \|\cdot\|_F)$, we note $\mathcal F(E,F)$ the set of mappings from $E$ to $F$ and we endow it with the family of semi-norms $(\|\cdot \|_{\mathcal B(y, r)})_{x \in E, r>0}$ defined for any $r>0$ and any $y_0\in \mathbb R^p$ as:
\begin{align*}
    \forall f \in  \mathcal F(\mathbb R^p): \ \
    \|f \|_{\mathcal B(y_0, r)} = \sup \{\|f(y)\|_F, \|y - y_0\|_E \leq r\}
\end{align*}
Let us start with a result of concentration of minimizing solutions to convex problems depending on a random matrix (more diverse and general settings are provided in Section~\ref{sec:concentration_of_the_solutions_to_conc_eq}).
\begin{theorem}
  Given a sequence $p_n \in \mathbb N^{\mathbb N}$ such that $p_n \leq O(n)$, a family of random vectors $X_n \in \mathcal M_{p_n,n}$ and a mapping $\Phi_n : \mathcal M_{p_n,n} \to \mathcal F(\mathbb R^{p_n},\mathbb R)$ such that there exist some constants $c, C_1, C_2, C_3>0$ and some parameters $\kappa>0, \sigma_n>0$ satisfying for any $n \in \mathbb N$:
\begin{itemize}
    \item $\forall M \in \mathcal M_{p_n,n}$, $\Phi_n(M)$ is a twice differentiable strictly convex mapping of bounded Hessian,
    \item with the classical order relation on the set of symmetric matrices, for all $M \in \mathcal M_{p_n, n}$ and $y \in \mathbb R^{p_n}$, $\restriction{d^2(\Phi_n(M))}{y}\geq \kappa_n$,
    \item there exists a deterministic vector $y_0\in \mathbb R^{p_n}$ such that a.s.:
    \begin{align*}
      \delta \equiv \max \left( \sigma \|X\|, \|\restriction{d\Phi(X)}{y_0}\| \right) \leq C_2
    \end{align*}
    \item for any $1$-Lipschitz mapping $f: \mathcal M_{p_n,n} \to \mathbb R$, and any $t>0$:
    \begin{align}\label{eq:concentration_X_main_result}
        \mathbb P \left( \left\vert f(X) - \mathbb E[f(X)] \right\vert \geq t \right) \leq C_1 e^{-ct^2}.
    \end{align}
    \item the mapping $M \mapsto (y \mapsto \restriction{d\Phi_n(M)}{y})$ is $C_3/\sqrt n$-Lipschitz from $(\mathcal M_{p_n,n}, \|\cdot \|_F)$ to $(\mathcal F(\mathbb R^{p_n}, \mathbb R^{p_n}), \|\cdot\|_{\mathcal B(y_0, \frac{\delta}{\kappa_n})})$
\end{itemize}
Then, introducing $Y_n \equiv \text{argmin}_{y\in\mathbb R^{p_n}} \Phi_n(X_n)(y) \in \mathbb R^{p_n}$, there exist two constant $c', C' >0$ such that for all $1$-Lipschitz mapping $f: \mathcal M_{p_n,n} \times \mathbb R^{p_n} \to \mathbb R$, and any $t>0$
    \begin{align}\label{eq:concentration_Y_main_result}
        \mathbb P \left( \left\vert f \left( \frac{\sigma_n}{\kappa_n}X_n, Y_n \right) - \mathbb E \left[  f \left( \frac{\sigma_n}{\kappa_n}X_n, Y_n \right) \right] \right\vert \geq t \right) \leq C' e^{-c'n (t/\kappa)^2} 
    \end{align}

  \end{theorem}
 In the third part of the paper this theorem is applied to show the concentration of the solution of the regularized robust regression defined for convex settings, we display below a rewriting of Corollary~\ref{cor:concentration_queue_Y}.
 
 \begin{corollary}\label{cor:concentration_solution_convex_problem_main_result}
   We consider a three differentiable convex mappings $h_i: \mathbb R \to \mathbb R$, $i\in[n]$, such that $\| h_i'\|_\infty,\| h_i{(2)}\|_\infty, \| h_i^{(3)}\|_\infty \leq \infty$, a sequence $p_n \leq O(n)$ and a sequence of random vectors $X_n=(x_i,\ldots,x_n) \in \mathcal M_{p_n,n}$ such that there exists three constants $K,C,c>0$ satisfying $\sup_{1\leq i \leq n} \| \mathbb E [x_i] \| \leq K $ and \eqref{eq:concentration_X_main_result} for any $1$-Lipschitz mapping $f: \mathbb R^p \to \mathbb R$. 
   If we note $Y_n\in \mathbb R^p$, solution to:
   \begin{align}
       \text{Minimize} \ \ \frac{1}{n} \sum_{i=1}^n h_i(x_i^T Y) + \|Y\|^2,  \ Y \in \mathbb R^p,
   \end{align}
   then for any $n \in \mathbb N$, and any $f : \mathbb R^{p_n} \to \mathbb R$, there exist some constants $c',C'>0$ such that:
   \begin{align*}
       \forall n \in \mathbb N, \forall t>0: \  \mathbb P \left(\left\vert f(Y_n) - \mathbb E[f(Y_n)] \right\vert \geq t \right) \leq  C' e^{-c'n t^2}+ C' e^{-c'n t^{2/3}}.
   \end{align*}
 \end{corollary}

Once the concentration of $Y$ is proven, one is led to try and estimate its first statistics.
For that purpose, one first needs to disentangle the contribution of each $x_i$ in\footnote{To give a arguably valid justification, we mention that if one can write $x_iY$ as a functional $\xi_i(x_iY_{-i})$ for $Y_{-i}$ independent of $i$, then one can resort to Stein identities to estimate $f(\xi_i(x_iY_{-i})x_i$.} $Y = \frac{1}{n}\sum_{i=1}^n f_i(x_iY)x_i$ thanks to the so-called ``leave-one-out'' technique that introduces the vector:
\begin{align}\label{eq:def_Y_m_i}
   Y_{-i} = \frac{1}{n} \sum_{\genfrac{}{}{0pt}{2}{j=1}{j \neq i}}^n f_i(a_j^TY_{-i})x_j,
\end{align} 
The link between $Y$ and $Y_{-i}$ is provided by the following formula that was already given in \cite{MAI19} but whose full justification (relying on concentration of measure and complex random matrix theory inferences taken from \cite{LOU21HV}) is an original contribution of this paper (the following result is a composition of Propositions~\ref{pro:point_fixe_x_iY} and~\ref{pro:lien_Y_Y_mi}).
\begin{proposition}\label{pro:lien_Y_Y_mi_main result}

  In the setting of Corollary~\ref{cor:concentration_solution_convex_problem_main_result}, for any $\delta \in \mathcal D_n$, let us note $\tilde Q^{\delta} \equiv (I_p - \frac{1}{n} \sum_{i=1}^n \mathbb E \left[ \frac{h_i^{(2)}(x_i^TY)}{1 - \delta_i h_i^{(2)}(x_i^TY)}  \right]\Sigma_i)^{-1}$ where $\Sigma_i \equiv \mathbb E[x_ix_i^T]$. The equation $\delta = \frac{1}{n} \diag_{i\in [n]}(\tr(\Sigma_i \tilde Q^\delta))$ has a unique solution that we note $\Delta$, and we have the concentration:
  \begin{align*}
    \forall t >0: \mathbb P \left( \left\vert x_i^TY - x_i^TY_{-i}-\Delta_i h_i'(x_i^TY) \right\vert\geq t \right) \leq C e^{- cnt^2/\log n} + C e^{-cn},
  \end{align*}
  for some constants $C,c>0$ independent of $n$.
\end{proposition}
This proposition leads us to introduce the mapping $\xi_i$ defined for any $z \in \mathbb R$ with the fixed point equation (see Proposition~\ref{pro:point_fixe_x_iY})
 \begin{align}\label{eq:point_fixe_xi}
    \xi_i(z) = h_i'(z + \Delta_i \xi_i(z)), &
    & \xi_i(z) \in \mathbb R,
  \end{align}
  then one can show (Proposition~\ref{pro:point_fixe_x_iY} again) that $h_i'(x_i^T Y) \in \xi_i(x_i^T Y_{-i})\pm \mathcal E_1 \left(\frac{1}{\sqrt n}\right)$. Finally, noting that $\xi_i'(t) = \frac{h_i^{(2)}(\zeta_i(t))}{ 1  - \Delta_i h_i^{(2)}(\zeta_i(t))}$ (where $\zeta_i(t)$ satisfies $\zeta_i(t) = t + \Delta_i h_i'(\zeta_i(t))$), one is lead to deduce the following result that could not be demonstrated (although it was validated with practical examples as one can see in next subsection).
\begin{conjecture}\label{conj:fixed_point_approximation}
  Considering sequence of Gaussian vectors $X_n = (x_1, \ldots, x_n)\in \mathcal M_{p_n,n}$ and convex twice differentiable mappings $h_i: \mathbb R \to \mathbb R$ satisfying the hypotheses of of Corollary~\ref{cor:concentration_solution_convex_problem_main_result}, if we note $m_i \equiv \mathbb E[x_i]$, $\Sigma_i = \mathbb E[x_ix_i^T]$, for all $n \in \mathbb N$, the unique vector $ m_Y$ satisfying:
  \begin{itemize}
    \item $m_Y = \sum_{i=1}^n \xi_i(m_i^T m_Y)\tilde Qm_i$
    \item $\tilde C = \sum_{i=1}^n \xi_i'(m_i^T  m_Y) C_i$
    \item $\tilde Q = (I_p - \tilde C)^{-1}$
    \item $\xi_i:\mathbb R \to \mathbb R$ is defined for any $z \in \mathbb R$ as the unique solution to $\xi_i(z) = h_i' \left( z + \frac{1}{n}\tr(C_i \tilde Q) \xi_i(z)  \right)$
  \end{itemize}
  Besides, if we note $\tilde \Sigma = \sum_{i=1}^n \xi_i(m_i^T  m_Y)^2 \Sigma_i$, the solution to the minimizing problem:
  \begin{align*}
      \text{Minimize:}\ \  \sum_{i=0}^n h_i(x_i^T Y) + \|y\|^2,&
      &\ y \in \mathbb R^p,
  \end{align*}  
  satisfies the estimation:
  \begin{align*}
      \left\Vert \mathbb E[Y ] -   m_Y\right\Vert \leq O \left(\frac 1{\sqrt n}\right)&
      &\text{and}&
      &\left\Vert \mathbb E[YY^T ] -  \tilde Q \Sigma \tilde Q\right\Vert_* \leq O \left(\frac 1{\sqrt n}\right)
  \end{align*}
\end{conjecture}

We were not able to prove this claim that however appears right on different practical examples. We give from Subsection~\ref{sse:estimation_gaussian_setting} and in the next subsections some reachable inferences in the Gaussian case, however this approach has two main issues:
\begin{itemize}
  \item we could not obtain the existence and uniqueness to the obtained fixed point equation
  \item the fixed point equation in question relies on the commutation of integrals $\mathbb E[\xi_i(z)]$ for $z \sim \mathcal N(m_z,v_z)$ and has therefore a very slow rate of convergence
\end{itemize}
Those two reasons make us think that the Gaussian approach is not the good one and that one has to find under the general assumption of concentrated vectors a supplementary equation (giving a link between $\mathbb E[\xi_i[x_i^T Y_{-i}]$ and $\mathbb E[Y]$ for instance) that will allow us to prove Conjecture~\ref{conj:fixed_point_approximation}. We give in next subsection some illustration of the validity of Conjecture~\ref{conj:fixed_point_approximation}.

\subsection{Application to the logistic regression}\label{sse:cas_X_gaussien}
To illustrate our theoretical results in a simple way we present the example of a supervised classification method called the ``the logistic regression'' already studied by \cite{MAI19}.
Considering a deterministic vector $m \in \mathbb R^p$ and a positive symmetric matrix $C \in \mathcal M_{p}$, we suppose we are given $n$ Gaussian random vectors $z_1,\ldots, z_n$, each one following the law $\mathcal N(y_im, C)$ where $y_1,\ldots, y_n \in \{-1,1\}$ are the ``labels'' of $z_i$ that determine the classes of our data. We assume for simplicity that the classes are balanced. Our classification problem aims at deducing from the \textit{training set} $z_1,\ldots, z_n$ and the labels $y_1,\ldots y_n$ a statistical characterization of our two classes that will allow us to classify a new coming data $x$, independent with the training set and following one or the other law. To introduce the problem, let us express the probability conditioned on a new data $x$, and knowing that $z_i$ is Gaussian, that $y_i = y$, for a given $y \in \{0,1\}$:
\begin{align*}
  \mathbb P(y_i = y \ | \ z_i) 
  &= \frac{\mathbb P(y_i = y)\mathbb P(z_i \ | \ y_i = y)}{\mathbb P(y_i = y)\mathbb P(z_i \ | \ y_i = y) + \mathbb P(y_i = -y)\mathbb P(z_i \ | \ y_i = -)}\\
  &= \frac{e^{-(z_i-ym)^TC^{-1}(z_i-ym)/2}}{e^{-(z_i-ym)^TC^{-1}(z_i-ym)/2}+e^{-(z_i+ym)^TC^{-1}(z_i+ym)/2}}\\
  &=\frac{1}{1+e^{-yz_i^TC^{-1}m}} = \sigma(yz_i^T \beta^*),
\end{align*}
noting $\sigma : t \mapsto 1/(1+e^{-t})$ and $\beta^* \equiv C^{-1}m$. The goal of the logistic regression is to try and estimate $\beta$ to be able to classify the data depending on the highest value of $\mathbb P(y_i = 1 \ | \ z_i) $ and $\mathbb P(y_i = -1 \ | \ z_i) $. For that purpose, we solve a regularized maximum likelihood problem:
\begin{align*}
  \min_{\beta \in \mathbb R^p} \frac{1}{p} \sum_{i=1}^n \rho(\beta^Tx_i) + \frac{\lambda}{2} \| \beta \|^2
\end{align*}
where $\rho(t) = \log(1+e^{-t})$, $x_i = y_i z_i$ and $\lambda>0$ is the regularizing parameter. Differentiating this minimizing problem, we obtain:
\begin{align*}
  \beta = \frac{1}{\lambda n} \sum_{i=1}^n f(x_i^T \beta) x_i,
\end{align*}
where $f : t \mapsto \frac{1}{1+e^t}$. If one chooses $\lambda$ sufficiently big, the assumptions of Corollary~\ref{cor:concentration_solution_convex_problem_main_result} are all satisfied, our results thus allow us to set the concentration $\beta \propto \mathcal E_2(1/\sqrt n) $ and estimate its first statistics.
The system of equations provided in Conjecture~\ref{conj:fixed_point_approximation} can be solved by successive iteration, we can then deduce the performances of the algorithm from the statistics of $Y$. We depicted on Figure~\ref{fig:soft_max_prediction} some of these predictions.
\\

\vspace{1cm}

  \begin{figure}[h]
  \centering
  \begin{tikzpicture}[font=\footnotesize]
    \renewcommand{\axisdefaulttryminticks}{4}
    \pgfplotsset{every axis legend/.append style={cells={anchor=west},fill=white, at={(0.7,1)}, anchor=north east, font=\scriptsize}}
    \begin{axis}[
    height=0.62\linewidth,
    width=1\linewidth,
      xmode=log,
      log basis x ={2},
      grid=major,
      ymajorgrids=false,
      scaled ticks=true,
      xlabel={Regularization parameter $\gamma$  \ \ $\rightarrow$ bias},
      ylabel={Misclassification rate}
      ]
      \addplot[only marks,mark = o,mark size=3pt,color=red!60!white,line width=1.5pt] plot coordinates{
      (0.00012207031250000008, 0.246)(0.00024414062500000016, 0.2447)(0.00048828124999999995, 0.2458)(0.0009765625, 0.2462)(0.001953125, 0.2512)(0.003906250000000001, 0.2464)(0.007812500000000002, 0.2454)(0.015625000000000007, 0.252)(0.03125, 0.2523)(0.0625, 0.2467)(0.12500000000000003, 0.2466)(0.25, 0.2578)(0.5, 0.2603)(1.0, 0.269)(2.0, 0.2657)(4.0, 0.2932)(8, 0.3125548726953468)
      };
      \addlegendentry{{ Empirical results class 1 }}
      \addplot[densely dashed,blue!60!white,line width=1pt] plot coordinates{
     (0.00012207031250000008, 0.25394)(0.00024414062500000016, 0.24963)(0.00048828124999999995, 0.24882)(0.0009765625, 0.24788)(0.001953125, 0.24681)(0.003906250000000001, 0.2465)(0.007812500000000002, 0.2464)(0.015625000000000007, 0.2499)(0.03125, 0.24706)(0.0625, 0.24637)(0.12500000000000003, 0.24733)(0.25, 0.24991)(0.5, 0.25603)(1.0, 0.26289)(2.0, 0.27348)(4.0, 0.29068)(8, 0.30824)
      };
      \addlegendentry{{ Theoretical results  class 1 } }
      \addplot[only marks,mark = x,mark size=3pt,color=red!60!white,line width=1.5pt] plot coordinates{
      (0.00012207031250000008, 0.08650000000000002)(0.00024414062500000016, 0.0887)(0.00048828124999999995, 0.08760000000000001)(0.0009765625, 0.09079999999999999)(0.001953125, 0.08150000000000002)(0.003906250000000001, 0.0877)(0.007812500000000002, 0.08930000000000005)(0.015625000000000007, 0.08789999999999998)(0.03125, 0.08540000000000003)(0.0625, 0.08579999999999999)(0.12500000000000003, 0.0826)(0.25, 0.08599999999999997)(0.5, 0.09260000000000002)(1.0, 0.09899999999999998)(2.0, 0.11680000000000001)(4.0, 0.13390000000000002)(8, 0.17172958735733101)
      };
      \addlegendentry{{ Empirical results  class 2}}
      \addplot[smooth,blue!60!white,line width=1pt] plot coordinates{
      (0.00012207031250000008, 0.08455999999999997)(0.00024414062500000016, 0.08716000000000002)(0.00048828124999999995, 0.08731999999999995)(0.0009765625, 0.08645000000000003)(0.001953125, 0.08675999999999995)(0.003906250000000001, 0.08689999999999998)(0.007812500000000002, 0.08645999999999998)(0.015625000000000007, 0.08797999999999995)(0.03125, 0.08565999999999996)(0.0625, 0.08777999999999997)(0.12500000000000003, 0.08733999999999997)(0.25, 0.08984000000000003)(0.5, 0.09345999999999999)(1.0, 0.10050999999999999)(2.0, 0.11434)(4.0, 0.13632)(8, 0.16113)
      };
      \addlegendentry{{ Theoretical results  class 2} }
    \end{axis}
  \end{tikzpicture}

   \caption{Misclassification rate for $p = 128$, $n = 512$,  $x_i \sim \mathcal N(y_i \mu, \Sigma)$, $\Sigma = 2 I_p$, $m_+ = (1,1,0,...)$, $m_- = - m_+$, for different regularizing parameters. Mean taken on $1000$ drawings for the estimation of the empirical misclassification error.}     
    \label{fig:soft_max_prediction}
    \end{figure}

Our work is organized as follow, we first provide the notation and basic results of the concentration of measure framework at the basis of our approach (more details are provided in Appendix~\ref{app:concentration_of_the_measure}). In a second section we present the theoretical results providing the concentration of solutions of contractive fixed point equations and we then apply those results in the third and last section to set the concentration and estimate the solution to some regularized robust regression problems.

\section{Concentration of Measure Theory framework}

We choose here to work with normed (or semi-normed) vector spaces, although CMT is classically developed in metric spaces and we adopt the viewpoint of \textit{Levy families} where the goal is to track the influence of the vector dimension over the concentration. Specifically, we are given a sequence of random vectors $(Z_p)_{p\geq \mathbb N}$ where each $Z_p$ belongs to a space of dimension $p$ (typically $\mathbb R^p$) and we want to obtain inequalities of the form:
\begin{align}\label{eq:inegalite_de_concentration}
  \forall p \in \mathbb N, \forall t>0 : \mathbb P \left(\left\vert f_p(Z_p) - a_p\right\vert\geq t\right) \leq \alpha_p(t),
\end{align}
where for every $p \in \mathbb N$, $\alpha_p : \mathbb R^+ \rightarrow [0,1]$ 
 and $f_p : \mathbb R^p \rightarrow \mathbb R$ is a $1$-Lipschitz function, and $a_p$ is either a deterministic variable (typically $\mathbb E[f_p(Z_p)]$) or a random variable (for instance $f_p(Z_p')$ with $Z_p'$ an independent copy of $Z_p$).
 The examples studied here will just admit exponentially decreasing $\alpha_p$ having the form $Ce^{-c(t/\sigma_p)^q}$ for $q>0$ or a sum of such terms. Since only the asymptotic behavior interests us ($p$ high), we will use the short notation $\mathcal E_q(\sigma_p)$ that is described below and allows to omit the constant term $C$ and $c$.



We define here two classes of concentrated vectors depending on the regularity of the class of sequences of functions $(f_p)_{p\in \mathbb N}$ satisfying \eqref{eq:inegalite_de_concentration}. 
When \eqref{eq:inegalite_de_concentration} holds for all the $1$-Lipschitz mappings $f_p$, $Z_p$ is said to be \textit{Lipschitz concentrated}; when true for all $1$-Lipschitz \textit{and linear} mappings $f_p$, $Z_p$ is said to be \textit{linearly concentrated}. This latter notion is weaker but still relevant in some practical settings (see Propositions~\ref{pro:concentration_X_y_y'}) where Lipschitz concentration can not be obtained with good concentration parameter; it is besides a sufficient assumption to control the norm (see Proposition~\ref{pro:tao_conc_exp}). We present in this section necessary notations and basic inferences result a more detailed presentation is provided in Appendix~\ref{app:concentration_of_the_measure}. 




\begin{defpro}\label{def:concentrated_sequence}
  Given a sequence of normed vector spaces $(E_p, \Vert \cdot \Vert_p)_{p\geq 0}$, a sequence of random vectors $(Z_p)_{p\geq 0} \in \prod_{p\geq 0} E_p$\footnote{A random vector $Z$ of $E$ is a measurable function from a probability space $(\Omega, \mathcal F, \mathbb P)$ to the normed vector space $(E, \|\cdot\|)$ (endowed with the Borel $\sigma$-algebra); one should indeed write $Z: \Omega \to E$, but we abusively simply denote $Z\in E$.}, a sequence of positive reals $(\sigma_p)_{p\geq 0} \in \mathbb R_+ ^{\mathbb N}$ and a parameter $q>0$, we say that $Z_p$ is Lipschitz \emph{$q$-exponentially concentrated} with an \emph{observable diameter} of order $O(\sigma_p)$ iff one of the two following equivalent assertions is satisfied:
  \begin{itemize}
     \item $\exists C,c>0 \ | \ \forall p \in \mathbb N, \forall \ 1\text{-Lipschitz} \ f : E_p \rightarrow \mathbb{R}, \forall t>0:$\footnote{one could replace below $\mathbb E[f(Z_p)]$ with $f(Z_p)$ or with $m_{f}$ where $Z_p'$ is an independent copy of $Z_p$ and $m_{f}$ is a median of $f(Z_p)$ (it satisfies $\mathbb P \left(f(Z_p) \geq m_{f}\right), \mathbb P \left(f(Z_p) \leq m_{f}\right) \geq \frac{1}{2}$)} 
     \begin{center}
      $\mathbb P \left(\left\vert f(Z_p) - \mathbb E[f(Z_p)]\right\vert\geq t\right) \leq C e^{-(t/c\sigma_p)^q}$, 
     \end{center}
     \item $\exists C,c>0 \ | \ \forall p \in \mathbb N, \forall \ 1\text{-Lipschitz} \ f : E_p \rightarrow \mathbb{R}, \exists \bar f \in \mathbb R \ | \ \forall t>0:$ 
     \begin{center}
      $\mathbb P \left(\left\vert f(Z_p) - \bar f\right\vert\geq t\right) \leq C e^{-(t/c\sigma_p)^q}$, 
     \end{center}
   \end{itemize} 
   We denote in this case $Z_p \propto \mathcal E_{q}(\sigma_p)$ (or more simply $Z \propto \mathcal E_{q}(\sigma)$). If $\sigma = O(1)$, we simply write $Z_p \propto \mathcal E_{q}$. 

   If the two upper inequalities are just satisfied by linear $1$-Lipschitz functionals, we say that $Z_p$ is linearly concentrated and we note $Z \in \mathcal E_q(\sigma)$. For this case it is relevent to introduce the notion of deterministic equivalent that is a sequence of deterministic vectors $\tilde Z_p \in \Pi_{p\geq 0} E_p$ satisfying:
  \begin{align*}
    \forall t>0: \ 
      \mathbb P \left(\left\vert f(Z_p ) - f(\tilde Z_p)\right\vert\geq t\right) \leq C e^{-(t/c\sigma_p)^q}.
  \end{align*}
  If this holds, we write $Z \in \tilde Z \pm \mathcal E_q(\sigma)$ and when we are just interested in the size of the deterministic equivalent, we employ the notation $Z \in O(\theta) \pm \mathcal E_q(\sigma)$ if $\|\tilde Z\|\leq O(\theta)$.
\end{defpro}
The equivalence between the two definition of the Lipschitz and linear concentration is proven in \cite{LOU20}, principally thanks to results issued from \cite{LED05}.
A simple but fundamental consequence of Definition~\ref{def:concentrated_sequence} is that the Lipschitz concentrated vector class is stable through Lipschitz transformation and the linearly concentrated vectors class is stable through affine transformation. The Lipschitz coefficient of the transformation controls the concentration (see Propositions~\ref{pro:stabilite_lipschitz_mappings} and~\ref{pro:stabilite_concentration_lineaire_affine}).

There exists a wide range of concentrated random vectors that can be found for instance in \citep{LED05}. We recall below some of the major examples of Lipschitz concentrated random vectors (therefore also linearly concentrated) that could all be taken as assumption for the results presented at the end of this paper. In the following theorems, we only consider sequences of random vectors of the vector spaces $(\mathbb R^p, \left\Vert \cdot\right\Vert)$. We will omit the index $p$ to simplify the readability of the results.
\begin{theorem}[Fundamental examples of concentrated vectors]\label{the:concentration_vecteur_gaaussien}
\begin{sloppypar} The following sequences of random vectors are concentrated and satisfy $ Z \propto \mathcal E_2$: 
\begin{itemize}
  \item $Z$ is uniformly distributed on the sphere $\sqrt p \mathbb S^{p-1}$.
  \item $Z \sim \mathcal N(0,I_p)$ has independent Gaussian entries.
  \item $Z$ is uniformy distributed on the ball $\sqrt p\mathcal B = \{x \in \mathbb R^p, \Vert x \Vert \leq \sqrt p\}$.
  \item $Z$ is uniformy distributed on $[0,\sqrt p]^p$.
  \item $Z$ has the density $d\mathbb P_Z(z) = e^{-U(z)}d\lambda_p(z)$ where $U : \mathbb R^p \rightarrow \mathbb R$ is a positive functional with Hessian bounded from below by, say, $cI_p$ with $c=O(1)$ and $d\lambda_p$ is the Lebesgue measure on $\mathbb R^p$.
\end{itemize}
\end{sloppypar}
\end{theorem}

A very explicit characterization of \textit{exponential concentration} is given by a bound on the different centered moments that must be always kept in mind.
\begin{proposition}\label{pro:characterization_moments}
  A random vector $Z \in E$ is $q$-exponentially concentrated with observable diameter of order $\sigma$ (i.e., $Z \propto \mathcal E_q(\sigma)$) if and only if there exists $C\geq 1$ and $c = O(\sigma)$ such that for any (sequence of) $1$-Lipschitz functions $f : E \to \mathbb R$:
  \begin{align}\label{eq:characterization_concentration_avec_moments}
    \forall r \geq q : \mathbb{E}\left[\left\vert f(Z) - f(Z')\right\vert^r\right] \leq C \left(\frac{r}{q}\right)^{\frac{r}{q}}c^r,
   \end{align} 
   where $Z'$ is an independent copy of $Z$. 
   Inequality~\eqref{eq:characterization_concentration_avec_moments} also holds if we replace $f(Z')$ with $\mathbb E[f(Z)]$ (of course the constants $C$ and $c$ might be slightly different).
   If $Z$ is only linearly concentrated then this inequality is only true for linear mappings $f$ and one can replace $f(Z')$ by $f(\tilde Z)$ for any $\tilde Z$ being a deterministic equivalent of $Z$.
\end{proposition}

We then provide two computation tricks that will be used several times in next sections.

\begin{lemma}[\cite{louart2021spectral}, Lemma~2]\label{lem:concentration_condtionee}
  In the setting of Definition~\ref{def:concentrated_sequence}, if we are given a sequence of events $\mathcal A_n$ such that $\mathbb P(\mathcal A_n) \geq O(1)$, then:
  \begin{align*}
    Z \propto \mathcal E_q(\sigma)&
    &\Longrightarrow&
    &(Z \ | \ \mathcal A) \propto \mathcal E_q(\sigma) .
  \end{align*}
\end{lemma}

\begin{lemma}\label{lem:tool_concentration_under_diying_events}
  Let us consider a random variable $X$ such that:
  \begin{align*}
    \forall t>0:&
    &\mathbb P( \left\vert X - X' \right\vert \geq t \ | \ \mathcal A_K) \leq C e^{-c(t/\sigma K^m)^q},
  \end{align*}
  for some constants $C,c, m, q>0$ and 
  where $(\mathcal A_K)_{K>0}$ is a family of events such that for some $K_0 >0$:
  \begin{align*}
    \forall K \geq K_0:& 
    &\mathbb P(\mathcal A_K^c) \leq Ce^{-cK^q\eta}
   \end{align*}
   then we have the concentration:
   \begin{align*}
     \mathbb P( \left\vert X - X' \right\vert \geq t\vert) \leq 2 C e^{-c(t/\sigma K_0)^q} + C e^{-c(t \eta^{m/q}/\sigma )^{q/{m+1}}}
   \end{align*}
\end{lemma}
\begin{proof}
  Let us first consider $t \in [0, K_0^{1+m} \sigma \eta^{\frac{1}{q}}]$, in that case $\eta^{\frac{1}{q}}K_0> t/(\sigma  K_0^m)$, we can then bound:
  \begin{align*}
  \mathbb P \left( \left\vert X - X' \right\vert \geq t  \right)
    \leq\mathbb P( \left\vert X - X' \right\vert \geq t \ | \ \mathcal A_{K_0}\vert) + \mathbb P(\mathcal A_{K_0}^c) 
    \leq 2C e^{-c(t/\sigma K_0)^q}
  \end{align*}

  When $t> K_0^{1+m} \sigma \eta^{\frac{1}{q}}$:
  \begin{align*}
    \mathbb P \left( \left\vert X - X' \right\vert \geq t  \right)
    &\leq \mathbb P( \left\vert X - X' \right\vert \geq t \ | \ \mathcal A_K\vert) + \mathbb P(\mathcal A_K^c)\\
    &\leq C e^{-c(t/\sigma K^m)^q} + C e^{-cK^q \eta}
  \end{align*}
  Then, if we chose $K \equiv \left( \frac{t}{\sigma \eta^{\frac{1}{q}}} \right)^{\frac{1}{m+1}} \geq K_0$ by hypothesis on $t$, we obtain the inequality:
  \begin{align*}
    \mathbb P \left( \left\vert X - X' \right\vert \geq t  \right) \leq C e^{-(t \eta^{m/q}/\sigma)^{q/2}}.
  \end{align*}
\end{proof}
The two last lemmas combined allows us to set the concentration of a product of several variables as presented in Proposition~\ref{pro:concentration_produit_variable} in the appendix.
The readers unfamiliar with concentration of measure inferences are warmly advised to consul Appendix~\ref{app:concentration_of_the_measure} providing basic tools to control the norm and deal with concatenation of concentrated vectors; we also provide some important results of concentration of random matrices and spectral convergence.

\section{Concentration of the solutions to convex problems and contractive fixed point equations} 
\label{sec:concentration_of_the_solutions_to_conc_eq}

This sections aims at setting the concentration  of the solution to convex problems like:
\begin{align*}
  \text{Minimize} \ \Psi(Y), Y \in \mathbb R^p,
\end{align*}
where $\Psi: \mathbb R^p \to \mathbb R$ is a random convex mapping. Although our results could be extended to cases where $\Psi$ is only differentiable, we will assume in all this section that $\Psi$ is $\mathcal C^2$ and strictly convex. This lower bound on the Hessian of $\Psi$ is important since it will appear in the denominator of the observable diameter of $Y$.

In practical uses, the randomness of $\Psi$ generally depends on a random vector $X \in F$ and the problem write:
\begin{align}\label{eq:first_convex_problem}
  \text{Minimize} \ \psi(X)(Y), Y \in \mathbb R^p
\end{align}
for $\psi : F \to \mathcal F(\mathbb R^p, \mathbb R^p)$ deterministic and $X\in F$ satisfying some concentration inequality. As said above, in practical examples $\restriction{d\psi(X)^2}{y}$ is assumed to have a lower bound independent of $X$. One might desire to also have a bound limiting the variability due to $X$. It is generally not true, however we study in the first subsection this first example since the concentration of $Y$ is then a simple consequence to the theorem of implicit functions and the observable diameter obtained is the same as the one that will be retrieved in more complex settings. 

Note of course that $Y$ minimizing \eqref{eq:first_convex_problem} is also the solution to the equation:
\begin{align*}
  \phi(X)(Y) = 0
\end{align*}
where we noted for all $A,y \in F \times \mathbb R^p$ $\phi(A)(y) \equiv \restriction{d\psi(A)}{y}$. We will consider this associated problem in next subsection.

\subsection{Bounded variability on $X$}

In this subsection, we place ourselves in the case $F = \mathbb R^p$ and to employ next theorem to random convex problems, one has to chose:
\begin{align*}
   \phi(X,Y) \equiv \restriction{d\psi(X)}{Y} \in \mathbb R^p
 \end{align*}
 to retrieve the setting of the introduction of this section. 
 Note that the partial derivatives $\frac{\partial \phi}{\partial X}$ and $\frac{\partial \phi}{\partial Y}$ respectively belong to $\mathcal M_{p,n}$ and $\mathcal M_{p}$. 
\begin{theorem}\label{the:concentration_convex_problem_variability_onX_bounded}
  Let us consider a (sequence of) random vectors $X \in \mathbb R^n$ satisfying $X \propto \mathcal E_q$ and a deterministic $\mathcal C^1$ mapping $\phi : X \in \mathbb R^n \times \mathbb R^p \to \mathbb R^p$ such that $\forall A,y \in X(\Omega) \times \mathbb R^p$:
  \begin{itemize}
    \item $\| \restriction{\frac{\partial \phi}{\partial X}}{(A,y)}  \| \leq\sigma$,
    \item $\restriction{\frac{\partial \phi}{\partial Y}}{(A,y)}$ is a positive symmetric matrix satisfying $ \restriction{\frac{\partial \phi}{\partial Y}}{(A,y)} \geq \kappa I_p$,
  \end{itemize}
  for some (sequence of) positive parameters $\sigma,\kappa >0$. Then, the unique solution to the equation:
  \begin{align*}
    \phi(X, Y) =0,&
    &Y \in \mathbb R^p
  \end{align*}
  satisfies the concentration inequality $ Y \propto \mathcal E_q(\frac{\sigma}{\kappa})$.
\end{theorem}
\begin{proof}
  The implicit function Theorem sets that there exists a $\mathcal C^1$ mapping $\theta: \mathbb R^n \to \mathbb R^p$ such that $Y = \theta(X)$ and:
  \begin{align*}
    \restriction{d\theta}{X} = \left( \restriction{\frac{\partial \phi}{\partial Y}}{(X, \theta(Y))} \right)^{-1} \restriction{\frac{\partial \phi}{\partial X}}{(X, \theta(Y))}.
  \end{align*}
  One can then directly deduce the concentration of $Y$ from the $O(\frac{\sigma}{\kappa})$-Lipschitz  character of $\theta$.
\end{proof}

In more complex settings like the problems of robust regression that will be studied in this paper, the quantity $\| \restriction{\frac{\partial \phi}{\partial X}}{(A,y)}  \|$ does not admit a bound valid for all the drawings of $X$. One then has to take into account the dispersion of the different drawings of $Y$ to be able to bound efficiently $\| \restriction{\frac{\partial \phi}{\partial X}}{(A,y)}  \|$. The solution presented here relies on a reformulation of convex problems as contractive fixed point equations more adapted to control the variability towards $X$.

\subsection{From convex problem to a contractive fixed point equation}

Starting from the equation
\begin{align*}
  \restriction{d\Psi}{Y} = 0
\end{align*}
one can introduce the mapping:
\begin{align*}
  \Phi: Y \longmapsto Y - \frac{\restriction{d\Psi}{Y}}{\|\restriction{d^2\Psi}{\cdot}\|_\infty },
\end{align*}
we see that if for all $Y\in \mathbb R^p$, $\restriction{d^2\Psi}{Y} \geq \kappa I_p$ then $\Phi$ is contractive since:
\begin{align*}
  \left\Vert \restriction{d\Phi}{Y} \right\Vert = \left\Vert I_p - \frac{\restriction{d^2\Psi}{Y}}{\|\restriction{d^2\Psi}{\cdot}\|_\infty} \right\Vert \leq 1 - \frac{\kappa}{\|\restriction{d^2\Psi}{\cdot}\|_\infty }
\end{align*}

We will therefore from now on consider solutions $Y \in \mathbb R^p$ to fixed point equations:
\begin{align*}
    Y = \Phi(Y),
  \end{align*}  
  where $\Phi$ is contractive.


We start with settings where $\Phi$ is affine and all the $\Phi^k(y)$ are linearly concentrated for deterministic $y$ (Subsection~\ref{sse:phi_affine}), then we study equation with general $\Phi$ with linearly and Lipschitz concentration hypotheses on $\Phi^k(y)$ for $k\leq O(\sqrt{\log n})$ (Subsection~\ref{sse:contractive_linear_lipschitz_phik}), finally we precise this result for mappings $\Phi$ that are Lipschitz concentrated for the infinite norm on balls (Subsection~\ref{sse:contractive_lipschitz}).

\subsection{Concentration of solutions to quadratic problems ($\Phi$ affine)}\label{sse:phi_affine}
Given two normed (or semi-normed) vector spaces $(E,\|\cdot\|_E)$, $(F,\|\cdot\|_F)$, we denote $\mathcal A(E, F)$ the set of continuous affine mappings from $E$ to $F$ and we endow it with the norm:
\begin{align*}
  \forall \phi \in \mathcal A(E,F):
  \|\phi\|_{\mathcal A(E,F) } = \|\mathcal L(\phi) \|_{\mathcal L(E,F) } + \|\phi(0)\|_ F
\end{align*}
where $\mathcal L(\phi) = \phi - \phi(0)$ and $\|\|\mathcal L(\phi) \|_{\mathcal L(E,F) } = \sup_{x\in E, \|x\|_E \leq 1} \|\mathcal L(\phi)(x)\|_F$.

\begin{example}\label{exe:}
  
Let us consider as an introducing example the quadratic problem:
\begin{align*}
  \text{Minimize:} \ \ \  \Psi(y) = \left\Vert \frac{1}{\sqrt n}X y - z \right\Vert^2 + \|y\|^2, \y \in \mathbb R^p
\end{align*}
for a deterministic vector $z \in \mathbb R^p$ and where we merely assume that the columns of $X = (x_1,\ldots, x_n)$ are independent and uniformly distributed on the sphere $\mathbb S^{p-1}$ (then since $X \propto \mathcal E_2$, and $\|X\|\leq \sqrt n$, we know from proposition~\ref{pro:stabilite_lipschitz_mappings} that $\frac{1}{\sqrt n}X \propto \mathcal E_2(1/\sqrt n)$  and besides $\|X/\sqrt n\| \leq 1$). 
If we note:
\begin{align*}
  \Phi: y \mapsto y -  \frac{1}{2} \left( \frac{1}{ n}XX^T y + \frac{1}{\sqrt n} Xz +y \right)
\end{align*}   
we know that:
\begin{align*}
  \left\Vert \mathcal L(\Phi) \right\Vert
  = \left\Vert I_p - \frac{1}{2} \left( \frac{1}{ n}XX^T  + I_p \right) \right\Vert
  \leq \frac{1}{2}
 \end{align*}
and for all $k \in \mathbb N$:
\begin{align*}
  \mathcal L(\Phi)^k(\Phi(0)) = \frac{1}{2^k}\left( I_p - \frac{1}{n}XX^T\right)^k  \in \mathcal E_q \left( 2^{-k}\right) \ \ \text{in} \  \ (E, \| \cdot \|),
\end{align*}
The Theorem~\ref{the:Concentration_linearire_solution_implicite_hypo_concentration_phi^k_pour tout_k} then allows us to show that $Y = \am_{y\in \mathbb R^p} \Psi(y)$ satisfies the linear concentration:
\begin{align*}
 Y \in \mathcal E_2 \left( \frac{1}{\sqrt n} \right).
\end{align*}         
\end{example}
Let us start with an simple result of linear concentration taken from~\cite{louart2022sharp}:
\begin{proposition}\label{pro:concentration_serie_variables_concentres}
  \sloppypar{Given two constants $q,r>0$, $\sigma_1,\ldots,\sigma_n \ldots \in (\mathbb R_+^{\mathbb N}) ^\mathbb N$, a (sequence of) reflexive normed vector spaces $(E, \|\cdot\|)$, $\tilde Z_1, \ldots, \tilde Z_n,\ldots \in E^{\mathbb N}$ deterministic, and $ Z_1, \ldots,  Z_n,\ldots \in E^{\mathbb N}$ random (possibly dependent) satisfying, for any $n\in \mathbb N$, $Z_n \in \tilde Z_n \pm \mathcal E_q(\sigma_n)$. If we assume that $Z \equiv\sum_{n \in \mathbb N} Z_n$ is pointwise convergent\footnote{For any $w \in \Omega$, $\sum_{n \in \mathbb N} \|Z_n(w)\| \leq \infty$ and we define $Z(w) \equiv \sum_{n \in \mathbb N}Z_n(w)$}, that $\sum_{n \in \mathbb N} \tilde Z_n$ is well defined and that $\sum_{n\in \mathbb N} \sigma_i \leq \infty$, then we have the concentration~:
  \begin{align*}
    \sum_{n \in \mathbb N} Z_n \in \sum_{n \in \mathbb N} \tilde Z_n \pm \mathcal E_q \left(\sum_{n \in \mathbb N} \sigma_n\right),
     \ \ \text{in } (E, \Vert \cdot \Vert),
  \end{align*}}
\end{proposition}
We can then deduce:
\begin{theorem}[Concentration of the resolvent]\label{the:Concentration_linearire_solution_implicite_hypo_concentration_phi^k_pour tout_k}
  Given a (sequence of) reflexive vector space $(E, \| \cdot \|)$, let $\Phi\in \mathcal A(E)$ be a (sequence of) random mapping such that there exist two (sequences of) parameters $\sigma >0$ and $\varepsilon \in (0,1)$, such that $\left\Vert \mathcal L(\Phi)\right\Vert \leq 1-\varepsilon$, and such that for all (sequences of) integer $k$:
  \begin{align*}
    &\mathcal L(\Phi)^k(\Phi(0)) \in \mathcal E_q \left(\sigma(1-\varepsilon)^k\right) \ \ \text{in} \  \ (E, \| \cdot \|)&.
  \end{align*}
  Then the random equation
  \begin{align*}
    Y = \Phi(Y)
  \end{align*}
  admits a unique solution satisfying the linear concentration:
  \begin{align*}
    Y = ( Id_E - \mathcal L(\Phi))^{-1} \Phi(0)\in \mathcal E_q \left( \frac{\sigma}{\varepsilon} \right).
  \end{align*}
\end{theorem}
\begin{proof}
  The vector $Y$ is well defined and expresses:
  \begin{align*}
    Y = ( Id_E - \mathcal L(\phi))^{-1} \phi(0) = \sum_{k=0}^\infty \mathcal L(\phi)^k \phi(0)
  \end{align*}
  Now Proposition~\ref{pro:concentration_serie_variables_concentres} allows us to conclude from the concentrations $\mathcal L(\phi)^k \phi(0) \in \mathcal E_q(\sigma(1-\varepsilon)^k)$ that:
  \begin{align*}
    \sum_{k=0}^\infty \mathcal L(\phi)^k \phi(0) \in \mathcal E_q \left( \sum_{k=0}^\infty \sigma(1-\varepsilon)^k  \right) = \mathcal E_q \left( \frac{\sigma}{\varepsilon} \right).
  \end{align*}
\end{proof}

With Theorem~\ref{the:Concentration_linearire_solution_implicite_hypo_concentration_phi^k_pour tout_k} at hand we can already prove the Linear concentration in the following theorem. The Lipschitz concentration will be provided later with Proposition~\ref{pro:lipschitz_COncentration_solution_conentrated_equation_hypothese_Psi_k_concentre}.
minimizing problem of the form
\begin{align*}
  \text{Minimize} \ \ y^TWy + \lambda Z y,& & y \in \mathbb R^p,
\end{align*}
\subsection{When $\phi^k(y)$ is linearly or Lipschitz concentrated for $k \leq \log(\eta_{(E, \| \cdot \|)})$ and $y$ deterministic and bounded.}\label{sse:contractive_linear_lipschitz_phik}

When we can not get a decreasing observable diameter for the iterates of $\phi$, or when $\phi$ is not affine, one needs a different approach that allows to treat, at the same time affine and non-affine mappings $\phi$ and extend very simply linear concentration inferences to Lipschitz concentration inferences. 
Given a normed vector space $(E, \Vert \cdot \|)$, we note $\mathcal F(E)$, the set of mappings from $E$ to $E$.
If $f$ is bounded, we denote $\| f \|_{\infty} = \sup_{x,y \in E} \|f(x)\|$.
 For a Lipschitz mapping $f \in F(E)$, we introduce the seminorm $\| \cdot \|_ {\mathcal L}$ which provides the Lipschitz parameter and will play the role of $\|\mathcal L(\Phi)\|$ in Proposition~\ref{pro:COncentration_solution_conentrated_equation_phi_affine_k_leq_log_eta}:
$$\| f \|_{\mathcal L} = \sup_{\genfrac{}{}{0pt}{2}{x,y \in E}{x \neq y}} \frac{\| f(x) - f(y)\|}{\|x - y \|}.$$
\begin{proposition}\label{pro:COncentration_solution_conentrated_equation_phi_affine_k_leq_log_eta}
  Given a (sequence of) reflexive\footnote{We suppose $E$ reflexive to be able to define an expectation operator on the set of random vectors of $E$ -- see \ref{app:definition_expectation}} vector space $(E, \| \cdot \|)$, we note $\eta$ its norm degree (of course $\eta \geq O(1)$), we then consider $\Phi\in \mathcal F(E, E)$, a (sequence of) random Lipschitz mapping such that there exists a (sequence of) parameters $\varepsilon \in (0,1)$ such that $\left\Vert \Phi\right\Vert_{\mathcal L} \leq 1-\varepsilon$. 
  Given an integer $k_0\in \mathbb N$ such that $k_0 \geq \lceil -\frac{\log(\eta)}{q\log(1 - 2 \varepsilon)}\rceil$, noting $\tilde Y$ the unique\footnote{The assumption on $\left\Vert \Phi\right\Vert_{\mathcal L}$ ensures the existence and uniqueness of $\tilde Y$.} (deterministic) solution to $\tilde Y = \mathbb E[\Phi^{k_0}(\tilde Y)]$, if we further assume that\footnote{The term $\sum_{i=0}^{k_0} (1- \varepsilon)^i = \frac{1}{\varepsilon} - \frac{1}{\varepsilon}(1- \varepsilon)^{k_0}$ will appear naturally in application as pictured in Lemma~\ref{lem:concentration_composition_intelligente} below. In the application we will see below this concentration is true for any $k\in \mathbb N$.}:
  \begin{align}\label{eq:concentration_phi^k_avec_norme_forte}
    &\Phi^{k_0}(\tilde Y) \in \tilde Y \pm \mathcal E_q \left(\sum_{i=0}^{k_0} (1- \varepsilon)^i \sigma\right)&
    &\text{in} \  \ (E, \| \cdot \|),
  \end{align} 
  for a given (sequence) of integer $\sigma >0$,
  then, the random equation
  \begin{align*}
    y = \phi(y), \ \ y \in E
  \end{align*}
  admits a unique solution $Y \in E$ satisfying the linear concentration:
  \begin{align*}
    Y \in \tilde Y \pm\mathcal E_q \left( \frac{\sigma}{\varepsilon}\right) .
  \end{align*}
\end{proposition}

\begin{proof}
  The mapping $y \mapsto \Phi(y)$ is contractive, that proves the existence and uniqueness of $Y$ ($E$ is complete since it is reflexive -- see Remark~\ref{rem:reflexif_complet}).

  Then let us try and bound $\|Y - \tilde Y\|$:
  \begin{align*}
    \left\Vert Y - \tilde Y\right\Vert
    & = \left\Vert  \Phi^{k_0}(Y) - \mathbb E \left[ \Phi ^{k_0}(\tilde Y) \right]\right\Vert \leq \left\Vert   \Phi^{k_0}(Y) - \Phi^{k_0}(\tilde Y)\right\Vert + \left\Vert  \Phi^{k_0}(\tilde Y) - \mathbb E \left[ \Phi ^{k_0}(\tilde Y) \right]\right\Vert\\
    & \leq \|\mathcal L(\Phi)\|^{k_0} \left \Vert Y - \tilde Y\right\Vert + \left\Vert  \Phi^{k_0}(\tilde Y) - \mathbb E \left[ \Phi ^{k_0}(\tilde Y) \right]\right\Vert.
  \end{align*}
  The inequality $\|\Phi^{k_0}\| \leq (1-\varepsilon)^{k_0}$, gives us the bound:
  \begin{align*}
    \left\Vert Y - \tilde Y\right\Vert
    & \leq \frac{1}{1-\left( 1 - \varepsilon \right)^{k_0}} \left\Vert  \Phi^{k_0}(\tilde Y) - \mathbb E \left[ \Phi ^{k_0}(\tilde Y) \right]\right\Vert.
  \end{align*}
  But we know from Proposition~\ref{pro:tao_conc_exp} that:
  \begin{align*}
    \left\Vert \Phi^{k_0}(\tilde Y) - \tilde \Phi_k(\tilde Y)\right\Vert \in O \left( \eta^{1/q} \sigma\sum_{i=0}^{k_0}\left( 1 - \varepsilon \right)^i \right) \pm \mathcal E_q(\eta^{1/q} \sigma)
  \end{align*}
  which allows us to conclude (since $\sum_{i=0}^k\left( 1 - \varepsilon \right)^i = \frac{1}{\varepsilon} - \frac{1}{\varepsilon}(1- \varepsilon)^k$) that: 
  $$\left\Vert Y - \tilde Y\right\Vert \in O \left( \frac{\eta^{1/q} \sigma}{\varepsilon} \right) \pm \mathcal E_q\left( \frac{\eta^{1/q} \sigma}{\varepsilon} \right).$$

  Returning to our initial goal (the linear concentration of $Y$), we now bound, for $f\in \mathcal L(E)$ (we still have $\|\mathcal L(\Phi)\|\leq 1- \varepsilon$):
  \begin{align*}
    \left\vert f(Y) - f(\tilde Y)\right\vert
    & \leq \left\vert f \left(  \Phi^{k_0}(Y) \right) - f \left(\Phi^{k_0}(\tilde Y)\right)\right\vert + \left\vert f \left(\Phi^{k_0}(\tilde Y) \right) - f \left(\mathbb E \left[ \Phi ^{k_0}(\tilde Y) \right]\right)\right\vert\\
    & \leq (1-\varepsilon)^{k_0}\left\Vert Y - \tilde Y\right\Vert + \left\vert f \left( \Phi^{k_0}(\tilde Y)  \right)-  \mathbb E \left[ f\left(\Phi ^{k_0}(\tilde Y) \right)\right]\right\vert.
  \end{align*}
  Further, noting that, with our choice of $k$, $(1 - \varepsilon)^{k_0}  = O(1/ \eta^{1/q})$, we conclude again from the concentration of $\Phi^{k_0}(\tilde Y)$ that:
  \begin{align*}
    f(Y) \in f(\tilde Y) \pm \mathcal E_q \left(\frac{\sigma}{\varepsilon}\right),
  \end{align*}
  thereby giving the sought-for concentration result.
\end{proof}

Let us extend the result of Proposition~\ref{pro:COncentration_solution_conentrated_equation_phi_affine_k_leq_log_eta}  to the case of Lipschitz concentration.

\begin{proposition}\label{pro:lipschitz_COncentration_solution_conentrated_equation_hypothese_Psi_k_concentre}
  In the setting of Proposition~\ref{pro:COncentration_solution_conentrated_equation_phi_affine_k_leq_log_eta}, if we additionally assume that we have the Lipschitz concentration:
  \begin{align*}
    \Phi^{k_0}(\tilde Y) \propto \mathcal E_q \left(\sum_{i=0}^{k_0} (1- \varepsilon)^i \sigma\right)&
    &\text{in} \  \ (E, \| \cdot \|).
  \end{align*}
  then we have the Lipschitz concentration:
  \begin{align*}
    Y \propto \mathcal E_q \left(\frac{\sigma}{\varepsilon}\right),
  \end{align*}  
\end{proposition}

\begin{proof}


  We already know from Proposition~\ref{pro:COncentration_solution_conentrated_equation_phi_affine_k_leq_log_eta} that $Y \in \mathcal E_q(\sigma / \varepsilon)$.
  To show the Lipschitz concentration of $Y$, let us consider a Lipschitz map $f : E \rightarrow \mathbb R$ and introduce the mappings:
  \begin{align*}
    &U :
    \begin{aligned}[t]
      E& &\longmapsto& &E \times \mathbb R \\
      y& &\longmapsto& &\left(y,f(y)\right)
    \end{aligned}
    &
    \hspace{0.9cm}
    &V :
    \begin{aligned}[t]
      E \times \mathbb R& &\longmapsto& &E  \\
      (y, t)\  & &\longmapsto& & y
    \end{aligned}
  \end{align*}
  (note that  $V \circ U = Id_E$).
  If we endow $E\times \mathbb R$ with the norm $\|\cdot \|_{\ell^\infty}$ satisfying $\forall (y,t) \in E\times \mathbb R$, $\|(y,t) \|_{\ell^\infty} = \max(\|y\|, |t |)$ then the mappings $U$ and $V$ are both $1$-Lipschitz, consequently $\|U\circ \Phi\circ V\|_{\mathcal L} \leq 1- \varepsilon$ and we can consider $\tilde Z \in E \times \mathbb R$ solution to:
  \begin{align*}
    \tilde Z = \mathbb E \left[ U\circ \Phi^{k_0}\circ V(\tilde Z) \right],
  \end{align*}
  by uniqueness of the solution to $y = \mathbb E[\Phi^{k_0}(y)]$, we know that $\tilde Z = (\tilde Y, \tilde t)$ with $\tilde t \equiv \mathbb E[ f(\Phi^{k_0}(\tilde Y))]$.
  Then:
  \begin{align*}
    (U\circ \Phi\circ V)^{k_0}(\tilde Z) = (\Phi^{k_0}( \tilde Y), f(\Phi^{k_0}( \tilde Y)))  \propto \mathcal E_q \left(\sum_{i=0}^{k_0} (1- \varepsilon)^i \sigma\right)& 
    &\text{in} \ (E \times \mathbb R, \| \cdot \|_{\ell^\infty})
   \end{align*} 
   by hypothesis. Therefore, one can deduce from Proposition~\ref{pro:COncentration_solution_conentrated_equation_phi_affine_k_leq_log_eta} the linear concentration of $Z$ solution to the fixed point equation:
  \begin{align}\label{eq:point_fixe_X}
    z = U \circ \Phi \circ V(z)&
    &z\in E \times \mathbb R;
  \end{align}
  from which one can deduce in particular, the concentration of the second component $f(Y)$.
\end{proof}

\subsection{When $\Psi \propto \mathcal E_2(\sigma)$ for the infinity norm}\label{sse:contractive_lipschitz}
If one can assume the stronger hypothesis that $\Psi$ is concentrated as a random mapping and for the infinity norm, then we can infer the concentration of all the iterates of $\Psi$. We give provide here the result for a general setting where the mapping $\phi$ is only bounded around a given point $y_0$, for that, for any $r>0$, we introduce the semi-norm $\|\cdot\|_{\mathcal B(y_0, r)}$ defined for any $f \in \mathcal F(E)$ and $y \in E$ as:
\begin{align*}
  \|f\|_{\mathcal B(y_0, r)} = \sup_{\|y- y_0\| \leq r} \|f(y)\|
\end{align*}
\begin{theorem}\label{the:lipschitz_COncentration_solution_conentrated_equation_Phi_concentre_norme_infinie}
Given a normed vector space $(E,\|\cdot \|)$ and a (sequence of) random mapping $\Phi\in \mathcal F^\infty(E)$ we assume that we are given a (sequence of) deterministic vector $y_0 \in E$ and three (sequences of) parameters $\sigma, \tau,\varepsilon>0$ such that: 
\begin{itemize}
  \item $\Psi$ is $(1-\varepsilon)$-Lipschitz
  \item $\|y_0 - \Psi(y_0)\| \leq \tau$
  \item $\Psi \propto \mathcal E_q \left(\sigma\right)$ in $\left( \mathcal F(E), \| \cdot \|_{\mathcal B(y_0, \tau/\varepsilon)} \right)$,
\end{itemize} 
then the random equation $Y = \Psi(Y)$ admits a unique solution $Y\in E$ satisfying the Lipschitz concentration:
  \begin{align*}
     Y \propto \mathcal E_q \left(\frac{\sigma}{\varepsilon}\right) .
  \end{align*}
\end{theorem}
One can see in the proof of the theorem below that the required concentration is actually the concentration of $\restriction{\Psi}{\mathcal B(y_0, \frac{\tau}{\varepsilon})}$ to be able to employ Proposition~\ref{pro:lipschitz_COncentration_solution_conentrated_equation_hypothese_Psi_k_concentre}.


The following two lemmas allows us to apply Proposition~\ref{pro:lipschitz_COncentration_solution_conentrated_equation_hypothese_Psi_k_concentre} to prove Theorem~\ref{the:lipschitz_COncentration_solution_conentrated_equation_Phi_concentre_norme_infinie} below.
\begin{lemma}\label{lem:stability_ball_contractive}
  Any $(1-\varepsilon)$-Lipschitz mapping $\psi: E\to E$ is stable on any ball $\mathcal B(y_0, \frac{1}{\varepsilon}\|y_0 - \psi(y_0)\|)$, where $y_0 \in E$.
\end{lemma}
\begin{proof}
  Let us note $\tau \equiv \|y_0 - \psi(y_0)\|$ and for any $y\in \mathcal B(y_0, \frac{\tau}{\varepsilon})$, let us bound:
  \begin{align*}
    \left\Vert \Psi(y) - y_0\right\Vert \ 
    &\leq \left\Vert \Psi(y) - \Psi(y_0)\right\Vert + \left\Vert \Psi(y_0) - y_0\right\Vert \\
    &\leq (1 - \varepsilon) \|y - y_0\| +  \|\Psi(y_0) - y_0 \| 
    \leq \ (1-\varepsilon)\frac{\tau}{\varepsilon} + \tau
    \ \ =  \ \frac{\tau}{\varepsilon}
  \end{align*}
\end{proof}
\begin{lemma}\label{lem:concentration_composition_intelligente} 
  Given a random mapping $\Psi : E \to E$, four (sequences of) parameters\footnote{Unlike in Theorem~\ref{the:lipschitz_COncentration_solution_conentrated_equation_Phi_concentre_norme_infinie}, $\epsilon$ can here tend to zero.} $\sigma, \varepsilon, \tau,\alpha>0$ and a (sequence of) deterministic vector $y_0 \in E$, if we assume that:
  \begin{itemize}
    \item $\|y_0 - \Psi(y_0)\| \leq \tau$,
     \item $\Psi$ is $(1 - \varepsilon)$-Lipschitz,
     \item $\Psi \propto \mathcal E_q(\sigma)$ in $( \mathcal F(E, E), \| \cdot \|_{\mathcal B(y_0,\frac{\tau}{\varepsilon})})$,
   \end{itemize}
  then for any integer $k>0$, we have the concentration:
  \begin{align*}
    \Psi^k \propto \mathcal E_q\left(\sum_{i=0}^{k} (1- \varepsilon)^i \sigma\right)&
    &\text{in} \ \ ( \mathcal F(E), \| \cdot \|_{\mathcal B(y_0,\frac{\tau}{\varepsilon})}),
  \end{align*}
\end{lemma}
\begin{proof}
  Thus, for any $f \in \Psi(\Omega)$ and for all $k \in \mathbb N$, $\Psi^k: \mathcal B(y_0,\frac{\tau}{\varepsilon}) \to \mathcal B(y_0,\frac{\tau}{\varepsilon})$, and we can bound for any supplementary mapping $g \in \Psi(\Omega)$:
  \begin{align*}
    \|f^k - g^k\|_{\mathcal B(y_0,\frac{\tau}{\varepsilon})}
    &\leq \sum_{i=1}^k (1-\varepsilon)^{i-1} \sup_{y \in \mathcal B(y_0,\frac{\tau}{\varepsilon})}\|f(g^{k-i}(y)) - g(g^{k-i}(y)) \| \\ 
    &\leq \sum_{i=0}^{k} (1- \varepsilon)^i \|f - g \|_{\mathcal B(y_0,\tau/\varepsilon)} .
   \end{align*}
   Thus the mapping $f \mapsto f^k$ is $\sum_{i=0}^{k} (1- \varepsilon)^i$-Lipschitz from $(\mathcal F(E), \| \cdot \|_{\mathcal B(y_0,\tau/\varepsilon)})$ to $(\mathcal F(E), \| \cdot \|_{\mathcal B(y_0,\tau)})$ and one can deduce the result of the Lemma from the concentration of $\Psi$.

\end{proof}

\begin{proof}[Proof of Theorem~\ref{the:lipschitz_COncentration_solution_conentrated_equation_Phi_concentre_norme_infinie}]
Noting $k_0 =\lceil -\frac{\log(\eta)}{q\log(1 - 2 \varepsilon)}\rceil$, let introduce $\tilde Y$ the solution of the fixed point equation $\tilde Y = \mathbb E[\Psi^{k_0}](\tilde Y)$. 
We know that:
\begin{align*}
  \left\Vert \tilde Y - y_0\right\Vert \ 
    &\leq \left\Vert \mathbb E \left[ \Psi^{k_0}(\tilde Y) - \Psi^{k_0}(y_0) \right]\right\Vert + \left\Vert \mathbb E \left[ \Psi^{k_0}(y_0) \right] - y_0\right\Vert\\
    &\leq (1 - \varepsilon)^{k_0} \|\tilde Y - y_0\| + \sum_{i= 0}^{k_0-1} (1- \varepsilon)^i \|\Psi(y_0) - y_0 \|
    \\ 
    &\leq \ \frac{\sum_{i=0}^{k_0} (1- \varepsilon)^i\tau}{1 - (1-\varepsilon)^{k_0}}
    \ = \ \frac{\tau}{\varepsilon},
\end{align*}

Therefore $\tilde Y \in \mathcal B(y_0, \frac{\tau}{\varepsilon})$ and we can besides deduce from the hypotheses of the theorem and Lemma~\ref{lem:concentration_composition_intelligente} that:
\begin{align*}
  \Psi^{k_0}(\tilde Y) \propto \mathcal E_q\left(\sum_{i=0}^{k} (1- \varepsilon)^i \sigma\right).
\end{align*}
  We then see that the hypotheses of Proposition~\ref{pro:lipschitz_COncentration_solution_conentrated_equation_hypothese_Psi_k_concentre} are satisfied and we can deduce the result of the theorem.
\end{proof}

\subsection{Concentration of minimizing solution $Y$ to convex problems $\Psi(X)(Y)$ with $\Psi$ deterministic and $X$ concentrated}
This study could be extended to a wide range and settings, but we now concentrate on a classical setting of Theorems~\ref{pro:COncentration_solution_conentrated_equation_phi_affine_k_leq_log_eta}-\ref{the:lipschitz_COncentration_solution_conentrated_equation_Phi_concentre_norme_infinie} where the randomness of $\Psi$ depends on a random vector $X \in F$ (for a normed vector space $(F, \| \cdot \|)$) and $Y$ is the minimizing solution of:
\begin{align*}
  \text{Minimize:} \ \  \phi(X)(Y)
\end{align*}
for a given deterministic mapping $\phi : F \to \mathcal F(\mathbb R^p, \mathbb R)$. 

The issue is then, not only to show the concentration of $Y$ but also the concentration of $(X,Y) \in F \times \mathbb R^p$ to be able to control the operations made on $X$ and $Y$ as needed in next section. 

\begin{corollary}\label{cor:concentration_lipschitz_solution_point_fixe}
  Given a reflexive vector space $(F, \|\cdot\|_F)$, a random vector $X \in F$ satisfying $X \propto \mathcal E_q$
  and a deterministic mapping $\phi : F \rightarrow \mathcal F(\mathbb R^p, \mathbb R^p)$, we assume that
   for all $X$, that there exists a (sequence) of parameters $\kappa >0$ such that $d^2(\phi(X)) \geq \kappa I_p$ (with the order relation on symmetric matrices).
  If we assume that there exists a (sequence of) vectors $y_0 \in {\mathbb R^p}$, such that:
  \begin{align}\label{eq:borne_X_dpsiX}
     \delta \equiv \max \left( \sigma \|X\|, \|\restriction{d\phi(X)}{y_0}\| \right)
   \end{align} and the mapping $A \mapsto d(\phi(A))$ is $O(\sigma)$-Lipschitz from $(F,\|\cdot\|_F)$ to\footnote{where for any $f \in \mathcal F(F,\mathbb R)$, $\|f\|_{F,B_{\mathbb R^p}(y_0, r)} = \sup \{\|f(y)\|_F, \|y - y_0\|_{\mathbb R^p} \leq r\}$} $(\mathcal F({\mathbb R^p}, {\mathbb R^p}), \|\cdot\|_{F,B_{(y_0,  \delta /\kappa)}})$,
  then, $Y = \text{argmin}_{y\in \mathbb R^p}\psi(X)(y)$ satisfies:
  \begin{align*}
     \ \ \left( \frac{\sigma}{\kappa}  X, Y \right) \propto \mathcal E_2 \left( \frac{\sigma}{\kappa} \right).
  \end{align*}

\end{corollary}
\begin{proof}
We endow the vector space $E \equiv F \times \mathbb R^p$ with the norm $ \|\cdot \|_{\ell^\infty}$ defined for any $(A,y) \in E$ as $\|(A,y)\|^{\ell^\infty} = \max(\|A\|_F,\|y\|)$. For any $K>0$, we place ourselves on the event $\mathcal A_K \equiv \{\|d^2\psi(X)\|_{\mathcal B(y_0, \frac{\delta}{\kappa})} \leq K\}$. Then for any $B \in F$, we introduce the mapping:
\begin{align*}
    \phi(B) : 
    \begin{aligned}[t]
      (F\times \mathbb R^p, \|\cdot \|^{\ell^\infty})& &\longrightarrow && (F\times \mathbb R^p, \|\cdot \|^{\ell^\infty}) \hspace{1.1cm}\\
      (A, y) \hspace{0.8cm}\hspace{0.1cm} & &\longmapsto& & \left(A - \frac{1}{K} (\sigma B - \kappa A) , y - \frac{1}{K} \restriction{d(\psi(B))}{y}\right).
    \end{aligned}
  \end{align*}

Noting $z_0 \equiv (0, y_0) \in F\times \mathbb R^p$, we can bound:
\begin{align*}
  \left\Vert z_0 - \phi(X)(z_0) \right\Vert
  = \max \left( \frac{\sigma}{K}\left\Vert X \right\Vert, \frac{1}{K}\left\Vert \restriction{d(\psi(X))}{y_0} \right\Vert \right)
  = \frac{\delta}{K},
\end{align*}
and of course, under $\mathcal A_K$, $\phi(X)$ is $(1-\frac{\kappa}{K})$-Lipschitz on $\mathcal B(z_0, \frac{\delta}{\kappa})$.

In addition, note that $B \mapsto \phi(B)$ is $O(\sigma/K)$-Lipschitz from $(X(\mathcal A_K), \|\cdot \|_F)$ to $(\mathcal F(E,E), \| \cdot \|^{\ell^\infty}_{\mathcal B(z_0,\delta /\kappa)})$, therefore, it satisfies:
\begin{align*}
  \phi(X) \ | \ \mathcal A_K\propto \mathcal E_q(\sigma/K) \ \ \text{in  } \ \ ( \mathcal F(\mathbb R^p), \| \cdot \|_{\mathcal B(z_0, \delta/ \kappa)}).
\end{align*}

The random vector $(\frac{\sigma}{\kappa} X, Y)$ being the unique solution to the fixed point equation $z = \phi(X)(z)$,
one can therefore employ Theorem~\ref{the:lipschitz_COncentration_solution_conentrated_equation_Phi_concentre_norme_infinie} to $\phi$ to set:
\begin{align*}
  \ \ \left( \frac{\sigma}{\kappa}  X, Y \right) \ | \ \mathcal A_K \propto \mathcal E_2 \left( \frac{\sigma}{\kappa} \right).
\end{align*}
 The concentration constants being independent of $K$, one can let $K$ tend to infinity to get the expected result.

\end{proof}

Let us now try to replace from the hypotheses of Corollary~\ref{cor:concentration_lipschitz_solution_point_fixe2} assumption \eqref{eq:borne_X_dpsiX} by a weaker assumption.
\begin{corollary}\label{cor:concentration_lipschitz_solution_point_fixe2}
under the hypotheses of Corollary~\ref{cor:concentration_lipschitz_solution_point_fixe}, if one further assumes that $\sigma \equiv \frac{1}{\sqrt p}$, $\eta_{F, \|\cdot\|} \leq O(\sqrt p)$ and replaces \eqref{eq:borne_X_dpsiX} with:

\begin{align*}
  \delta_0 \equiv  \max \left( \frac{1}{\sqrt p}\mathbb E[\|X\|], \mathbb E \left[ \left\Vert  \restriction{d\phi(X)}{y_0} \right\Vert \right] \right) \leq O(1),
\end{align*}
and that for all $\delta>\delta_0$, the mapping $A \mapsto d(\phi(A))$ is $O(\delta/\sqrt p)$-Lipschitz from $(F,\|\cdot\|_F)$ to $(\mathcal F({\mathbb R^p}, {\mathbb R^p}), \|\cdot\|_{F,B_{(y_0,  \delta /\kappa)}})$,
then one obtain the concentration:
\begin{align*}
  \left( \frac{1}{\sqrt p \kappa}  X, Y \right) \propto \mathcal E_2 \left( \frac{1}{\sqrt p\kappa} \right) + \mathcal E_1 \left( \frac{1}{p\kappa} \right).
\end{align*}
\end{corollary}
\begin{proof}
  We are going to use the bound resulting from Proposition~\ref{pro:tao_conc_exp}:
  \begin{align*}
    \forall \delta > 2 \delta_0 :&
    &\mathbb P \left( \frac{1}{\sqrt p}  \|X\| \geq \delta \right) \leq C e^{-c p \delta^2}&
    &\text{and}&
    &\mathbb P \left( \left\Vert  \restriction{d\phi(X)}{y_0} \right\Vert \geq \delta \right) \leq C e^{-c p\delta^2},
  \end{align*}
  for some constants $C,c>0$.
  Furthermore, placing ourselves on the event:
  \begin{align*}
    \mathcal A_\delta \equiv \left\{ \max \left( \frac{1}{\sqrt p}\|X\|,  \left\Vert  \restriction{d\phi(X)}{y_0} \right\Vert \right) \leq \delta \right\},
  \end{align*}
  for $\delta>0$ big enough, we know from Lemma~\ref{lem:concentration_condtionee} and Corollary~\ref{cor:concentration_lipschitz_solution_point_fixe} that $\left( \frac{1}{\kappa\sqrt p}  X, Y \right) \ | \mathcal A_\delta \propto \mathcal E_2(\delta/\sqrt p \kappa)$. Therefore there exist two constants $C',c'>0$ such that for any $1$-Lipschitz mapping $f: F\times \mathbb R^p \to \mathbb R$:
  \begin{align*}
    \mathbb P \left( \left\vert f \left( \left( \frac{1}{\kappa\sqrt p}  X, Y \right) \right) - \mathbb E \left[ f \left( \left( \frac{1}{\kappa\sqrt p}  X, Y \right) \right)  \right]\right\vert \geq t \ | \ \mathcal A_ \delta \right) \leq C'e^{- c'p\kappa^2t^2/\delta^2}
  \end{align*}
  Then we can deduce the concentration distinguishing the cases $t \leq \frac{\delta_0^2}{\kappa}$ and $t \geq \frac{\delta_0^2}{\kappa}$ and choosing in this last case $\delta = \sqrt{t/\kappa} \geq \delta_0$ (see Lemma~\ref{lem:tool_concentration_under_diying_events} for more precision).
\end{proof}



\section{Fixed point equation depending on independent data $x_1,\ldots, x_n$}\label{sse:estimation_point_fixe_complexe}
We study in this section the very common case of a matrix of data $X = (x_1,\ldots, x_n) \in \mathcal M_{p,n}$, where all the columns of $X$ are independent but not identically distributed and $\Psi$ acting on each column $x_i$ ``independently'' through the decomposition for all $A = (a_1,\ldots,a_n) \in \mathcal M_{p,n}$, all $y\in E$:
 \begin{align}\label{eq:decomposition_psi}
    \Psi(A)(y) = \frac{1}{n} \sum_{i=1}^n h(a_i^Ty) + \|y\|^2,
 \end{align}
 where $h_i : \mathbb R \to \mathbb R$ are twice differentiable and convex. This way $Y = \am \Psi(X)(y)$ is also solution to:
 \begin{align}\label{eq:fixed_point_equation_Y}
   Y = \frac{1}{n}\sum_{i=1}^n f_i(x_i^T Y) x_i,
 \end{align}
 where $f_i = -\frac{1}{2}h_i'$, and also to the contractive equation (since $I_p \leq d^2(\Psi(X)) \leq (1 + \frac{1}{n}\sup_{i\in[n]}\|h_i'{}'\|_{\infty})I_p) $:
 \begin{align*}
   Y = Y - \frac{Y - \frac{1}{n}\sum_{i=1}^n f(x_i^T Y) x_i}{1 + \frac{1}{n}\sup_{i\in[n]}\|h_i'{}'\|_{\infty} \|XX^T\|} 
 \end{align*}
 
 We will make the following assumptions:
 \begin{assumption}\label{ass:concentration_X}
  $X \propto \mathcal E_2$\footnote{In the initial Definition~\ref{def:concentrated_sequence}, we defined the concentration of a sequence of random vectors and here, $X_{n,p}$ is indexed by two natural numbers. A slight change of Definition~\ref{def:concentrated_sequence} allows us to adapt it to any set of indexes $S$ for $(X_s)_{s \in S}$ (in particular to $S = \mathbb N^2$), the two constants $C,c>0$ appearing in the concentration inequality are assumed to be valid for any $s \in S$ (i.e. for any $n,p \in \mathbb N^2$).}. 
\end{assumption}
\begin{assumption}\label{ass:x_i_independant}
  $X$ has independent columns $x_1,\ldots, x_n\in \mathbb R^p$\footnote{note that we do not assume that the $x_i$ are identically distributed as it is not required.}
\end{assumption}
\begin{assumption}\label{ass:n_p_commensurable}
  $p\leq O(n)$
\end{assumption}
Let us note for simplicity, for any $i \in [n]$:
\begin{align*}
  \mu_i \equiv \mathbb E[x_i]&
  &\Sigma_i \equiv \mathbb E[x_ix_i^T]&
  &\text{and}&
  &C_i \equiv \Sigma_i - \mu_i\mu_i^T
\end{align*}
We know from Proposition~\ref{pro:carcaterisation_vecteur_linearirement_concentre} that $\sup_{1\leq i \leq n}  \|C_i\| \leq O(1)$, but we also need to bound:
\begin{assumption}\label{ass:borne_norme_x_i}
  $\sup_{1\leq i\leq n}\|\mu_i\| \leq O(1)$ and $\inf_{1\leq i\leq n}\|\frac{1}{n}\tr \Sigma_i\| \geq O(1)$\footnote{This second hypothesis on the statistics of $X$ is not introduced to set the concentration of $Q$ but for the design of a deterministic equivalent. If the covariance of a vector $x_i$ is too small, one should be able to replace it by its expectation in the construction of a deterministic equivalent of $Q$, however, in this quasi asymptotic regime, it is not easy to identify the correct threshold, thus we prefer to place ourselves in the most common case where the energy of every data is taken into account in our estimation.}
\end{assumption}
 We will make the following assumptions on $f$, for reasons that will be clear later.
\begin{assumption}\label{ass:h_3_diff}
  The mappings $h_i: \mathbb R \to \mathbb R$ $1\leq i \leq n$ are three times differentiable.
\end{assumption}

\begin{assumption}\label{ass:h2_borne}
  $\sup_{i\in [n]}\|h_i'\|_\infty, \sup_{i\in [n]}\|h_i^{(2)}\|_\infty, \sup_{i\in [n]}\|h_i^{(3)}\|_\infty \leq O(1)$.
\end{assumption}

\subsection{Leave-one-out approach}
To estimate the expectation and covariance of $Y$, one needs to disentangle the influence of each data $x_i$ on $Y$. This leads us to studying the random vector $Y_{-i}$, defined in \eqref{eq:def_Y_m_i} and  independent with $x_i$ by construction.
To link $Y$ with $Y_{-i}$ we creates a ``bridge'' defined by a parameter $t \in [0,1]$, through a mapping $\Psi_{-i}^t :\mathcal M_{p,n} \to \lip(E)$, defined for any $A \in \mathcal M_{p,n}$ and any $y \in E$ with:
\begin{align*}
   \Psi_{-i}^t(A)(y)\equiv\frac{1}{n}\sum_{\genfrac{}{}{0pt}{2}{j=1}{j \neq i}}  f_i(a_j^Ty)a_j + t f_i(a_i^Ty)a_i.
 \end{align*} 
 Then, noting for any $t \in [0,1]$ $y_{-i}^A(t)$, the unique solution (when it exists) to $y_{-i}^A(t) = \Psi_{-i}^t(A)(y_{-i}^A(t))$, we see that:
 \begin{align*}
   Y_{-i} = y_{-i}^X(0)&
   &\text{and}&
   &Y = y_{-i}^X(1)
 \end{align*}
The next lemma sets the differentiability of the mapping $y_{-i}^A:[0,1] \to E$.

\begin{lemma}\label{lem:Y_differentiable}

  Under the differentiablility assumption on $h$ (Assumption~\ref{ass:h_3_diff}) the mapping $y^A_{-i}$ is differentiable and we have:
  \begin{align*}
    y^A_{-i}{}'(t) = \frac{\chi_A'(t)}{n} \left(I_p + \frac{1}{n} A_{-i} D_A(t)) A_{-i}^T\right)^{-1} a_i,
  \end{align*}
  where $\chi_A:t \mapsto t f_i(a_i^T y_{-i}^A(t))$ and $D_A : t \mapsto \diag(- f_i'(x_i^T y_{-i}^A(t))$.
\end{lemma}
\begin{proof}
  Starting from the fixed point equation satisfied by $Y$ in \eqref{eq:fixed_point_equation_Y}, we apply the inverse function theorem to the $\mathcal C^1$ bijective mapping:
  \begin{align*}
    \Theta : 
    \begin{aligned}[t]
      \mathbb R \times \mathbb R^p  && \longrightarrow && \mathbb R \times \mathbb R^p  \hspace{1.3cm} \\
      (t,y)\hspace{0.1cm}&&\longmapsto&& \left(t,y - \frac{1}{n}\sum_{\genfrac{}{}{0pt}{2}{j=0}{j \neq i}}f_i(a_j^Ty)a_j - \frac{t}{n} f_i(a_i^Ty)a_i\right).
    \end{aligned}
  \end{align*}
  It is indeed possible since $\mathbb R \times \mathbb R^p $ is a Banach space, $d\Theta$ is clearly bounded ($\Psi(A)$ is Lipschitz), and $\forall (t,y) \times (s,h) \in (\mathbb R \times \mathbb R^p )^2$:
  \begin{align*}
    d\Theta_{|(t,y)}\cdot (s,h) 
    &= \left(s,h -  \frac{s}{n} f_i(a_i^T y) a_i -   \frac{t}{n} f_i'(a_i^T y) a_i^T h a_i - A_{-i}\diag(f_i'(a_i^T y))_{i \in [n]}A_{-i} h  \right)\\
    &= \left(s,h -  \frac{s}{n} f_i(a_i^T y) a_i  - A_{-i}(t)\diag(f_i'(a_i^T y))_{i \in [n]}A(t) h  \right),
  \end{align*}
  if we note $A_{-i}(t) = (a_1,\ldots, a_{i-1},  t a_i,a_{i+1}, \cdots, a_n) \in \mathcal{M}_{p,n}$.
  Besides, $\forall t \in \mathbb R $, $f_i'(t) \leq 0$, thus $Q_{-i}^A(t) \equiv (I_p + A_{-i}(t) D_A(t) A)^{-1}$  is well defined and $d\Theta_{|(t,y)}$ is invertible with:
  \begin{align*}
    d\Theta_{|(t,y)}^{-1} = \left( \begin{array}{cc}
      1& 0\\
      Q_{-i}^A(t)f_i(a_i^Ty)a_i&Q_{-i}^A(t)
    \end{array}  \right).
  \end{align*}
  Therefore, $\Theta^{-1}$ is also $C^1$ and, we can differentiate $y_{-i}^A (t)= \Theta^{-1}(t,0)$ to obtain the identity:
  \begin{align}\label{eq:point_fixe_y'}
    y_{-i}^A{}'(t) =
    & - A_{-i} D_A(t) A_{-i}^T y_{-i}^A{}'(t)  \nonumber\\
    &\hspace{0.5cm}+  \frac{1}{n}f_i(a_i^Ty^A_{-i}(t))a_i 
    + \frac{t}{n} f'_i \left( a_i^Ty^A_{-i}{}(t) \right) a_i^Ty^A_{-i}{}'(t) a_i,
  \end{align}
  from which we retrieve directly the result of the Lemma.
\end{proof}
\subsection{Concentration of $t \mapsto y_{-i}^X{}'(t)$} 

Introducing, for any $t \in [0,1]$, and any $i \in[n]$, the diagonal matrix $D_{-i}(t) \in \mathcal D_{n}$ defined with
  \begin{align*}
    D_{-i}(t)_i = - tf'_i \left(x_i^Ty^X_{-i}(t) \right)&
    &\text{and}&
    &\forall j \in [n], j\neq i: D_{-i}(t)_j =  -f'_i \left(x_j^Ty^X_{-i}(t)\right)
  \end{align*}
  the random matrix $\restriction{d_2\Psi_{-i}^t(A)}{y}$ writes more simply $\frac{1}{n}XD_{-i}(t)X$ and we have the identity:
  \begin{align}\label{eq:formule_y'(t)}
    y_{-i}^X{}'(t) = \frac{1}{n}f_i(x_i^Ty_{-i}^X(t))Q_{-i}(t)x_i,
  \end{align}
  where, $Q_{-i}(t) \equiv Q(D_{-i}(t)) \equiv (I_p + \frac{1}{n}X D_{-i}(t)X^T)^{-1}$. Note in particular that:
  \begin{align}\label{eq:borne_Q_m_i}
    \forall t \in [0,1]: \ \ \|Q_{-i}(t)\| \leq 1
  \end{align}
  When $t=0$, the random matrix $Q_{-i}(0)$, that we note $Q_{-i}$ is then independent with $x_i$ like $D_{-i}\equiv D_{-i}(0)$, since then $[D_{-i}]_i = 0$.
 We are going to express the concentration of mappings defined on $[0,1]$ and having value on vectorial spaces ($\mathbb R^p$, $\mathcal{M}_{p,n}$ or $\mathcal D_n$). Those concentration are expressed with the infinite norm toward the parameter $t \in [0,1]$, for simplicity, we will superscript the vectorial norms with $\infty$ to designate those norms. For instance:
 \begin{align*}
   \left\Vert y_{-i}^X(\cdot) \right\Vert^\infty = \sup_{t \in [0,1]} \left\Vert y_{-i}^X(t) \right\Vert.
 \end{align*}
 We will also need to express norms on product of vectorial spaces. Given two normed vectorial spaces $(E, \|\cdot \|_E)$ and $(F, \| \cdot \|_F)$, we note $\|\cdot \|_E + \|\cdot \|_F$ the norm defined on $E \times F$ as:
 \begin{align*}
   \forall (x,y) \in E\times F: \ \ \ (\|\cdot \|_E + \|\cdot \|_F)(x,y) = \|x \|_E + \| y \|_F.
 \end{align*}

 We will now introduce a family of event indexed by the norm of $\|X\|$. Given $u>K$, we note:
  \begin{align*}
    \mathcal A_u \equiv \left\{ \|X\| \leq u \sqrt n \right\}
  \end{align*}
It is possible to bound $\|X\|$ thanks to the bound
\begin{align*}
  \|\mathbb E[X] \| \leq \sqrt n \sup_{1\leq i\leq n}\|\mathbb E[x_i]\| \leq O(\sqrt n)&
\end{align*}
Proposition~\ref{pro:tao_conc_exp} allows us to set $\mathbb E[\|X\|] \leq \|\mathbb E[X] \| + O(\sqrt n) \leq O(\sqrt n)$ and therefore $\|X\| \in O(\sqrt n) \pm \mathcal E_2$. 
  In other words
  \begin{align*}
    \forall u \geq \nu, \mathbb P \left( A_u^c \right) \leq C e^{-cnu^2},&
    &\text{where} \ \ \nu \equiv \frac{2 \mathbb E[\|X\|]}{\sqrt n}\leq O(1),
   \end{align*}
  and $C,c>0$ are two constants.

 \begin{proposition}\label{pro:concentration_X_y}
  For all $u \geq \nu$, and $i \in [n]$:
  \begin{align*}
  \left(\frac{1}{u^2\sqrt n} X, y_{-i}^X{}(\cdot)\right) \ | \ \mathcal A_u \propto \mathcal E_2 \left(\frac{u^2}{\sqrt n} \right) &
  &\text{in } \  \left(\mathcal M_{p,n} \times (\mathbb R^p)^{[0,1]}, \|\cdot\|_F^\infty+\|\cdot\|^\infty\right),  
  \end{align*}
 \end{proposition}
 \begin{proof}
 Given a (sequence of) parameter $u \geq K$, we first know from Lemma~\ref{lem:concentration_condtionee} that $(X \ |\ \mathcal A_u) \propto \mathcal E_2$
  We are going to employ Corollary~\ref{cor:concentration_lipschitz_solution_point_fixe} with the mapping $\Psi^{\cdot}_{-i}: \mathcal{M}_{p,n} \to \mathcal F((\mathbb R^p)^{[0,1]})$ that is defined for any $A,y,t \in \mathcal{M}_{p,n} \times \mathbb R^p \times [0,1]$ as:
  \begin{align}\label{eq:expression_psi_m_i_A}
    \Psi^{t}_{-i}(A)(y) = \Psi(A)(y) - \frac{1-t}{n}h(a_i^Ty),
  \end{align}
  where we recall that $\Psi(A)(y) = \frac{1}{n}\sum_{i=1}^n h(a_i^Ty)$. With this choice (recalling that $f = -\frac{1}{2} h'$), for all $A \in \mathcal{M}_{p,n}$, $y \in \mathbb R^p$:
  \begin{align*}
    \restriction{d(\Psi^t_{-i}(A))}{y} = \frac{1}{n}A_{-i}f_i(A_{-i}^Ty) + \frac{t}{n} f_i(a_i^Ty)a_i.
  \end{align*}


To apply Corollary~\ref{cor:concentration_lipschitz_solution_point_fixe}, we merely choose $y_0 : t \mapsto 0$, then for any $A \in X(\mathcal A_u)$:
  \begin{align*}
    \left\Vert \restriction{d(\Psi^\cdot_{-i}(A))}{y_0(\cdot)} \right\Vert^\infty
    = \frac{1}{n}\left\vert f_i(0) \right\vert \left\Vert A_{-i}\un + \frac{t}{n} a_i \right\Vert 
    \leq \|f\|_\infty u
  \end{align*}
  We can easily bound for any $A \in X(\mathcal A_u)$, $y \in (\mathbb R^p)^{[0,1]}$:
  \begin{align*}
    I_p \leq \restriction{d^2(\Psi^{\cdot}_{-i}(A))}{y} = - \frac{2}{n}A_{-i} \diag(f_i'(a_j^Ty))_{j\in [n]} A_{-i} - \frac{2t}{n}f_i'(a_i^Ty) a_ia_i^T + I_p \leq \|f'\|_\infty u^2
  \end{align*}
  (recall that $\forall t\in \mathbb R$, $f_i'(t) \leq 0$).

  Besides for any $y \in \mathcal B(y_0, u \|f\|_\infty)$ and any $A,B \in X(\mathcal A_u) \subset \mathcal{M}_{p,n}$:
  \begin{align*}
    \|\restriction{d(\Psi(A))}{y} - \restriction{d(\Psi(B))}{y}\|^\infty
    &\leq  \frac{1}{n} \left\Vert  (A - B)f_i(A^Ty)\right\Vert^\infty + \frac{1}{n}\left\Vert B \left( f_i(A^Ty) - f_i(B^Ty)\right)\right\Vert^\infty\\
    &\leq \frac{1}{n}(|f_i(0)| + \|A\| \|y\|^\infty \|f'\|_\infty )\| A-B\| + \frac{1}{n}\|B \| \|A-B\| \|y\|^\infty \|f'\|_\infty \\
    &\leq  O \left(\frac{u^2}{\sqrt n}\right) \|A-B\|;
  \end{align*}
  and the same way, for any $a,b \in x_i(\mathcal A_u) \subset \mathbb R^p$ we show that:
  \begin{align*}
    \frac{1}{n}\|f_i(a^Ty)a - f_i(b^Ty)b\|^\infty
    &\leq  O \left(\frac{u^2}{\sqrt n}\right) \|a-b\|,
  \end{align*}

  Therefore, noting $E \equiv (\mathcal R^p)^{[0,1]}$, we deduce from \eqref{eq:expression_psi_m_i_A} that for all $u> K$, $A \mapsto \restriction{d(\Psi^{\cdot}_{-i}(A))}{\cdot}$ is $O \left(\frac{u^2}{\sqrt n}\right)$-Lipschitz from $(X(\mathcal A_u), \|\cdot \|_F)$ to $(\mathcal F(E,E), \|\cdot\|_{\mathcal B(y_0, u)}^\infty)$.
  We can then deduce from Corollary~\ref{cor:concentration_lipschitz_solution_point_fixe} the result of the proposition.


\end{proof}

\begin{corollary}\label{cor:concentration_queue_Y}
  $Y(\cdot) \propto \mathcal E_2(n^{-1/2}) + \mathcal E_{\frac{2}{3}}(n^{-3/2})$.
\end{corollary}
\begin{proof}

  Noting that $Y$ is a $1$-Lipschitz transformation of $(\frac{1}{u^2\sqrt n} X, y_{-i}^X{}(\cdot))$, let us bound thanks to Proposition~\ref{pro:concentration_X_y} for any $u\geq \nu$ and any independent copy of $Y$, $Y'$:
  \begin{align*}
    \mathbb P \left( \left\vert f_i(Y) - f_i(Y')\right\vert \geq t \ | \ \mathcal A_u\right) 
    & \leq C e^{-cnt^2/u^4}&
    &\text{and}&
    &\mathbb P \left( \mathcal A_u^c \right)\leq C e^{-cnu^2}.
  \end{align*}
  One can then apply Lemma~\ref{lem:tool_concentration_under_diying_events} with $\sigma = 1/\sqrt n$, $m=2$ and\footnote{See Example~\ref{exe:norm_degree} to get more precision} $\eta = \eta_{(\mathcal{M}_{p,n}, \| \cdot \|)} = O(1/n)$ to obtain the result of the corollary.
\end{proof}

 Let us now prove the concentration of $y_{-i}^X{}'$ that we will integrate in next subsection.
 \begin{proposition}\label{pro:concentration_X_y_y'}
  Under $\mathcal A_\nu$, $\forall i \in [n]$, $y_{-i}^X{}'(\cdot) \in \mathcal E_2 \left(\frac{1}{\sqrt n} \right) $ 
  in $((\mathbb R^p)^{[0,1]}, \|\cdot\|^\infty)$.
 \end{proposition}
 One could propose as in Corollary~\ref{cor:concentration_queue_Y} a sharper control on the tail of $y_{-i}^X(\cdot)$ but the expression becomes complicated, so we prefer not to present it.
 \begin{proof}
  One could once again employ Corollary~\ref{cor:concentration_lipschitz_solution_point_fixe}, but it is straightforward here to employ the explicit formula given by \eqref{eq:formule_y'(t)}:
  \begin{align*}
    y_{-i}^X{}'(t) = \frac{1}{n}f_i(x_i^Ty_{-i}^X(t)) Q_{-i}(t)x_i,
  \end{align*}
  that allows us to state that $y_{-i}^X{}'(\cdot)$ is concentrated as a $O(1/\sqrt n)$-Lipschitz transformation under $\mathcal A_\nu$ of the concentrated vector $\left( X, \sqrt n  Y_{-i}(\cdot),D_{-i}(\cdot) \right)$ (recall indeed that $\|Q_{-i}(t)\|, \|f_i\|_\infty, \|X\|/\sqrt n, \|D_{-i}\|^\infty \leq O(1)$).
 \end{proof}

 \subsection{Integration of $\frac{\partial y_{-i}^X(t)}{\partial t}$}\label{sse:integration_y_mi_prime}
Now that the concentration of the objects $y_{-i}^X(t)$ and $y_{-i}^X(t){}'$ are well understood, we are able to integrate the formula provided by Lemma~\ref{lem:Y_differentiable} to express the link between $Y$ and $Y_{-i}$. We just give some preliminary results to control the matrix $Q_{-i}(\cdot)$. In a first time, let us study $Q_{-i} \equiv Q_{-i}(0) = Q_{-i}(D_{-i})$ which is independent with $x_i$.

  \begin{proposition}\label{pro:Q_m_i_proche_de_Q_m_i_cdot}
    Under $\mathcal A_\nu$, $ \forall i\in [n]$:
    \begin{align*}
      \left\Vert Q_{-i}(\cdot)x_i - Q_{-i}x_i\right\Vert \in O \left(\sqrt{\log(n)}\right) \pm \mathcal E_2 &
      &\text{in} \ \ (\mathbb R^{[0,1]}, \|\cdot\|^\infty)
    \end{align*}
  \end{proposition}
  This proposition is a consequence of the following lemma.
  \begin{lemma}\label{lem:borne_D_m_D_m_i}
    Under $\mathcal A_\nu$, $\sup_{t\in[0,1]}\|D_{-i}-D_{-i}(t)\|_F \leq O(1)$
  \end{lemma}
  \begin{proof}
    
  From identities $D_{-i}(t) \equiv \diag(f_i'(X^T y_{-i}^X(t)))$ and:
  \begin{align*}
    X^T y_{-i}^X(t) = \frac{1}{n}X^TX_{-i}f_i(X^T y_{-i}^X(t)) + \frac{t}{n}X^Tx_if_i(x_i^Ty_{-i}^X(t)),
  \end{align*}
  we can bound, under $\mathcal A_\nu$ (recall that $\|X\|,\|x_i\| \leq O(\sqrt n)$ and $\|f\|_\infty \leq O(1)$):
  \begin{align}\label{eq:borne_D_D_m_i}
    \|D_{-i}-D_{-i}(t)\|_F 
    &\leq \|f'{}'\|_\infty \|X^T y_{-i}^X(0) -X^T y_{-i}^X(t)\| \nonumber\\
    &\leq \frac{\|f'{}'\|_\infty}{\varepsilon}\frac{t}{n} \left\Vert X^Tx_if_i(x_i^Ty_{-i}^X(t))\right\Vert
    \ \leq  \ O \left( 1 \right)
  \end{align}
  \end{proof}
  \begin{proof}[Proof of Proposition~\ref{pro:Q_m_i_proche_de_Q_m_i_cdot}]
  Under $\mathcal A_\nu$ and for any $t\in[0,1]$, let us bound:
  \begin{align*}
    \left\Vert (Q_{-i}(t)- Q_{-i})x_i\right\Vert 
    &\leq \frac{1}{n}  \left\Vert Q_{-i}(t)X_{-i} (D_{-i}-D(t)) X_{-i}^TQ_{-i}x_i\right\Vert \\
    &\leq O \left(\frac{1}{\sqrt n}\right)  \|X_{-i}^TQ_{-i}x_i\|_\infty,
  \end{align*}
  thanks to Lemma~\ref{lem:borne_D_m_D_m_i} and 
  since $\|X\| \leq O(\sqrt n)$ and $\|Q_{-i}(t)\|\leq 1$.

  Now, it is possible to show (see \cite[Lemma D.3]{LOU21HV}) that  $X_{-i}^TQ_{-i}x_i\in \mathcal E_2(\sqrt n)$. Besides, one can bound thanks to the independence between $x_i$ and $Q_{-i}$ ($\mathcal A_\nu$ being overwhelming, it preserves in a sense the independence relations):
  \begin{align*}
    \left\Vert \mathbb E \left[ X_{-i}^TQ_{-i}x_i \ | \ \mathcal A_\nu\right] \right\Vert_\infty
    &\leq \left\Vert \mathbb E \left[ X_{-i}^TQ_{-i} \mathbb E[x_i] \right] \right\Vert_\infty + O \left( \frac{1}{n} \right)\\
    &\leq  \mathbb E \left[ \left\Vert X_{-i} \right\Vert\right] \| \mathbb E[x_i]\| \leq O(\sqrt n),
  \end{align*}
  derived from Assumption~\ref{ass:borne_norme_x_i} and \eqref{eq:borne_Q_m_i}. Thus, Proposition~\ref{pro:tao_conc_exp} allows us to bound:
  \begin{align*}
    \mathbb E \left[\left\Vert  X_{-i}^TQ_{-i}x_i \right\Vert_\infty \ | \ \mathcal A_\nu \right] \leq O(\sqrt{\log(n)}),
  \end{align*}
  from which we deduce the result of the proposition.
\end{proof}

We have now all the elements to prove:
\begin{proposition}\label{pro:lien_Y_Y_mi}

  7 6.8
  Under $\mathcal A_\nu$, $\forall i\in[n]$:
  \begin{align*}
    \left\Vert Y - Y_{-i}-\frac{1}{n}f_i(x_i^TY)Q_{-i}x_i\right\Vert \in O \left( \frac{\sqrt{\log n}  }{n} \right) \pm \mathcal E_2 \left(\frac{1}{n}\right) 
  \end{align*}
\end{proposition}
\begin{proof}
  Setting $\chi(t) \equiv tf_i(x_i^T y^X_{-i}(t)) \in \mathbb R$, let us integrate between $0$ and $t$ the identity $y_{-i}^X{}'(t) = \chi'(t) \frac{1}{n} Q_{-i}(t) x_i$:
  \begin{align}\label{eq:integree}
    y^X_{-i}(t)-Y_{-i}  = \frac{1}{n} f_i(x_i^TY)Q_{-i}x_i + \frac{1}{n}\int_0^t \chi'(u)  (Q_{-i}(u)- Q_{-i}(0) )x_i du .
  \end{align}
  Now, $\chi'(u) = f_i(x_i^Ty^X_{-i}(u)) + t f_i'(x_i^Ty^X_{-i}(u))x_i^Ty^X_{-i}{}'(u)$, and we know from Proposition~\ref{pro:concentration_X_y} and Proposition~\ref{pro:concentration_X_y_y'} that under $\mathcal A_\nu$:
  \begin{itemize}
    \item $f_i(x_i^T y_{-i}^X(\cdot) ) \in O(1) \pm  \mathcal E_2$ and $\|f_i(x_i^T y_{-i}^X(\cdot) )\|_\infty \leq O(1)$
    \item $f_i'(x_i^T y_{-i}^X(\cdot) ) \in O(1) \pm  \mathcal E_2$ and $\|f_i'(x_i^T y_{-i}^X(\cdot) )\|_\infty \leq O(1)$
    \item $x_i^T y_{-i}^X{}'(\cdot) \in O(1) \pm  \mathcal E_2$ and $\|x_i^T y_{-i}^X{}'(\cdot) \|_\infty \leq O(1)$
  \end{itemize}
  therefore $\chi'(\cdot) \in O(1) \pm \mathcal E_2$ in $\left(\mathbb R^{[0,1]}, \|\cdot\|_\infty \right)$, $\|\chi'\|_{\infty} \leq O(1)$ and one can bound:

  \begin{align*}
    &\left\Vert \chi'(\cdot)  (Q_{-i}(u)- Q_{-i}(0) )x_i \right\Vert \\
    &\hspace{2cm}\leq \left\vert \chi'( \cdot)\right\vert \left\Vert (Q_{-i}(\cdot)- Q_{-i}(0) )x_i \right\Vert\in O(\sqrt{\log n}  ) \pm \mathcal E_2,
  \end{align*}
  in $(\mathbb R^{[0,1]}, \|\cdot \|^\infty)$. Since the integration between $0$ and $t$ is $1$-Lipschitz for the infinite norm on $[0,t]$, we have the concentration:
  \begin{align*}
    \left\Vert  \frac{1}{n}\int_0^t \chi'(u)  (Q_{-i}(u)- Q_{-i}(0) )x_i du\right\Vert \in O(\sqrt{\log n}  ) \pm \mathcal E_2.
  \end{align*}
\end{proof}
Note first that $Y$ and $Y_{-i}$ have comparable first statistics.
\begin{corollary}\label{cor:statistique_Y_Y_m_i}
  \begin{align}\label{eq:statistique_Y_pareil_statistique_Y_m_i}
\left\{\begin{aligned}
  &\left\Vert \mathbb E_{\mathcal A_\nu} \left[Y\right] -\mathbb E_{\mathcal A_\nu} \left[Y_{-i}\right]\right\Vert \leq O \left(\sqrt{\frac{\log n}{n}}\right)\\
  &\left\Vert \mathbb E_{\mathcal A_\nu} \left[YY^T\right] -\mathbb E_{\mathcal A_\nu} \left[Y_{-i}Y_{-i}^T\right]\right\Vert_* \leq O \left(\sqrt{\frac{\log n}{n}}\right),
\end{aligned}\right.
\end{align}
where we recall that $\|\cdot \|_*$ is the nuclear norm satisfying for any $A \in \mathcal{M}_{p,n}$, $\|A\|_* = \tr ( \sqrt{AA^T})$.
\end{corollary}
\begin{proof}
It is just a consequence of Proposition~\ref{pro:lien_Y_Y_mi} and the bounds:
  \begin{align*}
  &\left\Vert \mathbb E_{\mathcal A_\nu} \left[\frac{1}{n}f_i(x_i^TY)Q_{-i}x_i  \right]\right\Vert\leq \frac{1}{n}\mathbb E_{\mathcal A_\nu} \left[ |f_i(x_i^TY)| \|x_i\|\|Q_{-i}\|\right]\leq O \left(\frac{1}{\sqrt n}\right)\\
  & \left\Vert \mathbb E_{\mathcal A_\nu} \left[\frac{1}{n^2}f_i(x_i^TY)^2Q_{-i}x_ix_i^T Q_{-i}  \right]\right\Vert_*
  \leq \frac{1}{n^2}\mathbb E_{\mathcal A_\nu} \left[ |f_i(x_i^TY)^2| \|x_i\|^2\|Q_{-i}\|^2\right] \leq O \left(\frac{1}{n}\right)\\
  & \left\Vert \mathbb E_{\mathcal A_\nu} \left[\frac{1}{n}f_i(x_i^TY)^2Q_{-i}x_iY_{-i}^T  \right]\right\Vert_*
  \leq O \left(\frac{1}{\sqrt n}\right),
\end{align*}
\end{proof}

One can then wonder why we set a result as complex as Proposition~\ref{pro:lien_Y_Y_mi} if it was simply to obtain these simple relations between the first statistics of $Y$ and $Y_{-i}$. It is because the behaviors of $Y$ and $Y_{-i}$ diverge when one looks at projections on $x_i$.

The observable diameter of order $O(\sqrt{\log n}/n)$ in Proposition~\ref{pro:lien_Y_Y_mi} allows us to keep good concentration bounds when $Y$ is multiplied on the left by $x_j^T$, $j \in [n]$ (indeed, under $\mathcal A_\nu$, $\|x_j\|\leq O(\sqrt n)$). This time, the term $\frac{1}{n}f_i(x_i^TY)x_j^TQ_{-i}x_i$ can be of order $O(1)$ in particular when $j=i$. 
For all $ j \in [n]$, $\frac{1}{n}f_i(x_i^TY)x_jQ_{-i}x_i \in \mathcal E_2(1/\sqrt n)$ (as a $1/\sqrt n$ transformation of $(X, \sqrt n y_{-i}^X(\cdot))$) thus if $j\neq i$:
\begin{align*}
  \mathbb E_{\mathcal A_\nu} \left[\frac{1}{n}f_i(x_i^TY)x_jQ_{-i}x_i \right] \leq O \left(\frac{1}{\sqrt n}\right)
\end{align*}
but when $j=i$ this quantity can be of order $O(1)$, therefore we obtain the concentrations:
\begin{corollary}\label{cor:concentration_xi_Y_xi_Y_m_i}
  \begin{align}\label{eq:x_j_Y_fonction_x_j_Y_m_i}
  \left\{
  \begin{aligned}
    &x_j^T Y - x_j^T Y_{-i} \in O\left(\sqrt{\frac{\log n}{ n}}\right) \pm \mathcal E_1 \left(\frac{1}{\sqrt n}\right) \ \ \text{when} \  \ j \neq i\\
    &x_i^T Y - x_i^T Y_{-i} - \frac{1}{n} x_i^T Q_{-i}x_i f_i(x_i^TY) \in O\left(\sqrt{\frac{\log n}{ n}}\right) \pm \mathcal E_1 \left(\frac{1}{\sqrt n}\right) .
  \end{aligned}\right.
\end{align}
\end{corollary}

\subsection{Implicit relation between $x_i^T Y$ and $x_i^T Y_{-i}$}\label{sse:implicit_relation_x_iY_x_i Y_mi}

The linear concentration \eqref{eq:x_j_Y_fonction_x_j_Y_m_i} interests us particularly because it allows us to replace in the identity
\begin{align*}
  Y = \frac{1}{n} \sum_{i=1}^n f_i(x_i^TY) x_i,
\end{align*}
the quantity $f_i(x_i^TY) x_i$ with a quantity $f_i(\zeta_i(x_i^TY_{-i})) x_i$ (for a given mapping $\zeta_i : \mathbb R \to \mathbb R$) that is more easy to manage thanks to the independence between $x_i$ and $Y_{-i}$. 
The random variable is estimated thanks to a fixed point equation already introduced, in the case of Wishart matrices, in \cite{SIL95, PAJ09,louart2021spectral}. Introducing the notation
\begin{align*}
    \tilde Q^{\Lambda(\tilde D)}(\tilde D) = \left( I_p + \frac{1}{n} \sum_{i=1}^n \mathbb E \left[ \frac{\tilde D_i \Sigma_i}{1+\Lambda_i(\tilde D) \tilde D_i} \right] \right)^{-1},
  \end{align*} 
  where, for all $\Gamma \in \mathcal{D}_{n}$, we note $\Lambda(\Gamma)$ the unique solution to:
  \begin{align*}
   \forall i \in [n]: \ \Lambda_i(\Gamma) = \frac{1}{n}\tr \left( \Sigma_i \tilde Q^{\Lambda(\Gamma)}(\Gamma) \right),
  \end{align*}
  (see \cite[Theorem 1]{louart2021spectral} for existence and uniqueness).

That allows us to express the following theorem (to keep a simple expression, we chose to take here $D \in \mathcal D_n^+$ and $Q = (I_p + \frac{1}{n}XDX^T)^{-1}$, recall however that in the studied example $D$ has negative entries).
\begin{theorem}[\cite{LOU21HV}, Theorem 8.1, Theorem 8.2]\label{the:concentration_resolvente_avec_diagonale_aleatoire}
Given a random diagonal matrix $D \in \mathcal D_n$ with positive entries and a random matrix $X = (x_1,\ldots, x_n) \in \mathcal M_{p,n}$ satisfying Assumptions~\ref{ass:concentration_X},~\ref{ass:x_i_independant} and~\ref{ass:borne_norme_x_i}, in the regime $p\leq O(n)$ and under the assumptions:
   \begin{itemize}
     \item for all $i \in[n]$, there exists a random diagonal matrix $D^{(i)}$, independent of $x_i$, such that $ \sup_{i\in[n]}\|D_{-i} - D_{-i}^{(i)}\|_F \leq O(1)$,
     \item there exist three constants $\kappa,\kappa_D, \varepsilon>0$ ($\varepsilon \geq O(1)$ and $\kappa,\kappa_D\leq O(1)$), such that $\|X\|,\|Y\| \leq \sqrt n\kappa$, $\|D\| \leq \kappa_D$,
   \end{itemize}
   we have the linear concentration:
  \begin{align*}
    Q  \in Q^{\Lambda(D)}(D) \pm \mathcal E_1 \left(\log(n)\right) + \mathcal E_{\frac{1}{2}}(1/\sqrt n)  &
    &\text{in} \ (\mathcal M_{p}, \|\cdot \|_F),
  \end{align*}
 \end{theorem} 
 For simplicity, we will note from now on:
\begin{align*}
  \Delta \equiv\Lambda(D),
\end{align*}
(where the expectation in the definition of $\tilde Q^{\Lambda(D)}(D)$ is taken on $\mathcal A_\nu$).
One can then deduce in particular from this theorem that for any $i \in [n]$:
 \begin{align}\label{eq:concentration_autour_de_delta}
   \frac{1}{n}x_i^T Q_{-i}x_i \ | \ \mathcal A_\nu \in \Delta_i \pm \mathcal E_1 \left(\log(n)\right) + \mathcal E_{\frac{1}{2}}(1/\sqrt n)
 \end{align} 
It sounds then natural to introduce the equation
\begin{align}\label{eq:point_fixe_XY}
  z = x_i^T Y_{-i} + \Delta_i f_i(z), &
  &z \in \mathbb R
\end{align}
whose solution is close to $x_i^T Y$ as stated by next proposition. 
\begin{proposition}\label{pro:point_fixe_x_iY}
  Given $i \in [n]$, and $x \in \mathbb R$, the equation:
  \begin{align}\label{eq:point_fixe_z'}
    z = x + \Delta_i f_i(z), &
    & z \in \mathbb R,
  \end{align}
  admits a unique solution that we note $\zeta_i(x)$, one then has the approximation:
  \begin{align*}
     x_i^T Y \in \zeta_i(x_i^T Y_{-i})\pm \mathcal E_1 \left(\frac{1}{\sqrt n}\right) 
   \end{align*} 
\end{proposition}
To prove this proposition, we are going to employ the following lemma which is just an adaptation of the result of Lemma~\ref{lem:stability_ball_contractive} concerning contractive mappings to the case of convex mappings.
 \begin{lemma}\label{lem:borne_point_fixe}
  Given a twice differentiable convex mapping $f: \mathbb R^p \to \mathbb R$ such that for all $y \in \mathbb R^p$, $ \restriction{d^2f}{y} \geq \kappa$ and a vector $y_0 \in \mathbb R^p$ such that $\|\restriction{df}{y_0}\| \leq \tau$, we know that the fixed point $y* = \min_{y \in \mathbb R^p}f_i(y)$ also satisfies: 
  \begin{align*}
    \|y* - y_0\| \leq \frac{\tau}{\kappa}
  \end{align*}
  \end{lemma}
  \begin{proof}
    One just has to introduce $K \equiv \sup_{y\in \mathcal B(y_0, \frac{\tau}{\kappa})} \| \restriction{d^2f}{y} \|$, then the mapping  $\phi: y \mapsto y - \frac{1}{K}\restriction{df}{y}$ is $(1 - \frac{K}{\kappa})$-Lipschitz on $\mathcal B(y_0, \frac{\tau}{\kappa})$ and one can conclude with Lemma~\ref{lem:stability_ball_contractive}.
  \end{proof}
\begin{proof}[Proof of Proposition~\ref{pro:point_fixe_x_iY}]
As the solution of the minimizing convex problem ($\Delta_i>0)$:
\begin{align*}
  \text{Minimize:} \ \phi(z) = \|z\|^2 + \Delta_i h(z) - x_{i}^T Y_{-i}z, 
\end{align*}
$\zeta(x_{i}^T Y_{-i}z)$ is well defined and unique.
  Now, we can bound under $\mathcal A_\nu$:
  \begin{align*}
    \left\Vert  \restriction{d\phi}{x_i^T Y} \right\Vert 
    =\left\Vert  x_i^T Y + \Delta_i f_i(x_i^T Y) - x_{i}^T Y_{-i} \right\Vert \in O \left( \sqrt {\frac{\log n}{n}} \right) \pm \mathcal E_2 \left( \frac{1}{\sqrt n} \right)
  \end{align*}
  thanks to Corollary~\ref{cor:concentration_xi_Y_xi_Y_m_i}. We can therefore employ Lemma~\ref{lem:borne_point_fixe} to deduce the result of the proposition.
\end{proof}
We end this subsection with a little result that will allow us to differentiate $\zeta_i$ and:
\begin{align*}
  \xi_i : t \mapsto f_i(\zeta_i(t))
\end{align*}
\begin{lemma}\label{lem:Derive_zeta_borne}
  Given $i\in [n]$, the mapping $\zeta_i$ is differentiable and we have the identity:
  \begin{align*}
    \xi_i'(t) = \frac{f_i'(\zeta_i(t))}{ 1  - \Delta_i f_i'(\zeta_i(t))}.
  \end{align*}
\end{lemma}
\begin{proof}
  Considering $z,t \in \mathbb R$, let us express:
  \begin{align*}
    \zeta_i(z + t) - \zeta_i(z) = t + \Delta_i \left(f_i(\zeta_i(z + t)) - f_i(\zeta_i(z))\right)
  \end{align*}
  thus $|\zeta_i(z + t) - \zeta_i(z)| \leq \frac{t}{1 - \Delta_i \|f'\|_\infty} $ (note that it implies that $\zeta_i$ is continuous). 
  Let us bound:
  \begin{align*}
     &\left\vert f_i(\zeta_i(z + t)) - f_i(\zeta_i(z)) - f_i'(\zeta_i(z))(\zeta_i(z + t) -\zeta_i(z))\right\vert \\
     &\hspace{2.5cm} \leq \|f'{}'\|_\infty \left\vert \zeta_i(z + t) -\zeta_i(z)\right\vert^2 \leq \frac{t^2 \|f'{}'\|_\infty}{\left(1 - \Delta_i \|f'\|_\infty\right)^2}
   \end{align*} 
   Dividing the upper identity by $t$ we can bound:
   \begin{align*}
     \left\vert \frac{1}{t} \left(\zeta_i(z + t) - \zeta_i(z)\right)  -1 - \frac{\Delta_i}{t} f_i'(\zeta_i(z))(\zeta_i(z + t) -\zeta_i(z))\right\vert \leq \frac{t \|f'{}'\|_\infty\Delta_i}{\left(1 - \Delta_i \|f'\|_\infty\right)^2}.
   \end{align*}
   We can then let $t$ tend to $0$ to conclude that $\zeta_i$ is differentiable and we obtain the identity:
   \begin{align*}
     \zeta_i'(t) = 1 + \Delta_i f_i'(\zeta_i(t)) \zeta_i'(t).
   \end{align*}
   and one easily deduce from this formulation of $\zeta'_i$ the expression of $\xi_i'$
\end{proof} 

The formulation of $\xi_i'$ leads us to introducing the notation
\begin{align*}
   \tilde Q \equiv (I_p - \frac{1}{n}\sum_{i=1}^n\mathbb E[\xi_i'(x_i^T Y_{-i})] \Sigma_i)^{-1}&
 \end{align*} 
 that satisfies the estimation:
 \begin{corollary}\label{cor:estimation_tQ_xi_concentration_Q}
   $\left\Vert \tilde Q - \tilde Q^{\Delta}(\mathbb E[D]) \right\Vert_F \leq O \left( \frac{\log n}{\sqrt n} \right)$
 \end{corollary}
 The proof can be done with the same diagonal matrix $\Delta$ but we think it would be more interesting to wait for the final definition of $\Delta$ to show the stronger result.

\subsection{Expression of the mean and covariance of $Y$.}\label{sse:estimation_gaussian_setting}
Introducing the notation $\forall i \in[n]$, $\xi_i = f \circ \zeta_i$, next Proposition gives us first estimations of the deterministic objects:
\begin{align*}
  m_Y \equiv \mathbb E[Y]&
  &\text{and}&
  &\Sigma_Y \equiv  \mathbb E[YY^T],
\end{align*}
\begin{proposition}\label{pro:estimation_moyenne_covariance_Y}
Noting $\check Y = \frac{1}{n} \sum_{i=1}^n\xi_i(x_i^T Y_{-i})x_i$, we can approximate:
\begin{align*}
   \left\Vert Y - \check Y\right\Vert \in O \left(\frac{\log n}{\sqrt n}\right) + \mathcal E_1 \left(\frac{\log n}{\sqrt n}\right)
 \end{align*} 
 and we can estimate $\left\Vert m_Y-  \mathbb E \left[\check Y\right]\right\Vert , \left\Vert \Sigma_Y -  \mathbb E \left[\check Y \check Y^T\right]\right\Vert_* \leq O \left(\frac{\log n}{\sqrt n}\right)$.
\end{proposition}
\begin{proof}
 
Let us bound:
\begin{align*}
  \left\Vert Y - \check Y\right\Vert  
  &\leq \left\Vert  \frac{1}{n} \sum_{i=1}^n \left( f_i(x_i^T Y) -f_i(\zeta_i(x_i^T Y_{-i})) \right)x_i \right\Vert 
  \leq O \left(\sup_{1\leq i \leq n}\left\vert x_i^TY - \zeta_i(x_i^T Y_{-i})\right\vert\right)
\end{align*}
 Besides, we know from Propositions~\ref{pro:point_fixe_x_iY} that $(|x_i^TY - \zeta_i(x_i^T Y_{-i})|)_{1\leq i \leq n} \in \mathcal E_1 \left(\frac{\log n}{\sqrt n}\right) $ in $(\mathbb R^{n}, \| \cdot \|_\infty)$, thus Proposition~\ref{pro:tao_conc_exp} implies that:
\begin{align*}
   \sup_{1\leq i \leq n}\left\vert x_i^TY - \zeta_i(x_i^T Y_{-i})\right\vert 
  \in O \left(\frac{\log n}{\sqrt n}\right) + \mathcal E_1 \left(\frac{\log n}{\sqrt n}\right),
\end{align*}
from which we deduce the first result of the proposition. The estimation of the expectation and the non-centered covariance of $Y$ is a direct consequence, indeed, for the covariance, note that for any deterministic matrix $A \in \mathcal{M}_{p} $ such that $\|A\| \leq 1$:
\begin{align*}
  \left\vert \tr \left( A \left( \Sigma_Y -  \mathbb E \left[\check Y\check Y^T\right] \right) \right) \right\vert 
  &\leq\left\vert \mathbb E \left[ Y^T AY -  \check Y^TA\check Y\right]  \right\vert 
   =  \left\vert \mathbb E \left[(Y - \check Y)^T A (Y + \check Y) \right]  \right\vert\\ 
  &\leq O \left( \|A\| \mathbb E[\|Y - \check Y \|] \right) \leq O \left( \frac{\log n}{\sqrt n} \right)
\end{align*}
\end{proof}

  \subsection{Computation of the estimation of the mean and covariance of $Y$ when $X$ is Gaussian}
The estimation given by Proposition~\ref{pro:estimation_moyenne_covariance_Y} becomes particularly interesting when $X$ is Gaussian because in that case, the random variable $z_i \equiv x_i^T Y_{-i}$ is also Gaussian (when all the random vectors $x_j$ are fixed, for $j \neq i$) and admits the statistics:
\begin{align*}
  \mathbb E_i[z_i] = m_i^T Y_{-i}&
  &\text{and}&
  & \mathbb E_i[z_i^2] = Y_{-i}^T\Sigma_i Y_{-i}
\end{align*}
(where we recall that $m_i \equiv \mathbb E[x_i]$ and $\Sigma_i \equiv \mathbb E[x_ix_i^T]$). The estimation of quantities of the form $\mathbb E_{j} \left[ \xi(z_j) u^Tx_j\right]$ will then be done in two steps that will be justified by Lemmas~\ref{lem:Stein} and~\ref{lem:integration_moyenne_et_covariance_concentres} below:
\begin{enumerate}
  \item ``separate'' with Stein-like identities, the ``functional part'' $\xi(z_j)$ from the ``vectorial part'' $u^Tx_j$ in $\mathbb E_j \left[ \xi(z_j) u^Tx_j \right]$,
  \item show that the randomness brought by $Y_{-i}$ in $z_i$ can be neglected so that it can be approximated by a Gaussian random variable $\tilde z_i \sim \mathcal N(\mu_i,\nu_i)$ with:
  \begin{align}\label{eq:definition_nu_i_mu_i}
  \mu_i \equiv m_i^Tm_Y&
  &\text{and}&
  &\nu_i \equiv \tr(\Sigma_Y\Sigma_i) - \mu_i^2,&
  &i \in[n].
\end{align} 
\end{enumerate}

\begin{proposition}\label{pro:first_approximation_m_C_Y}
  Introducing the quantities:
\begin{align*}
  \tilde m_Y^{(1)} = \frac{1}{n} \sum_{i=1}^n \mathbb E[\xi_i(\tilde z_i)] \tilde Qm_i& 
  &\text{and}&
  &\tilde C_Y^{(1)} = \frac{1}{n} \sum_{i=1}^n \mathbb E[\xi_i(\tilde z_i)^2] \tilde Q\Sigma_i \tilde Q \nonumber
\end{align*}
we have the estimations $\|m_Y - \tilde m_Y^{(1)} \|, \|C_Y - \tilde C_Y^{(1)} \|_* \leq O \left( \frac{\log n}{\sqrt n} \right)$.
\end{proposition}
The estimation merely rely on two lemmas. The first one is a derivation of the Stein identity:
\begin{lemma}\label{lem:Stein}
  Given a Gaussian vector $x \sim \mathcal N(\mu, C)$ for $\mu \in \mathbb R^p$ and $C \in \mathcal M_{p}$ positive symmetric, two deterministic vectors $w,u \in \mathbb R^p$, and a deterministic matrix $A \in \mathcal M_{p,n} $, we have the identities:
    \begin{align*}
    \mathbb E [f_i(w^T x) u^T x] 
    &= \mathbb E [f_i(w^T x) ]  u^T\mu +  \mathbb E [ f_i'(w^T x) ]u^TC w\\
    \mathbb E [f_i(w^T x)x^TAx] 
    &= \mathbb E [f_i(w^T x) ]  \tr \left(A(\mu\mu^T + C)\right) + \mathbb E [ f_i'(w^T x) ]w^T C\left(A +A^T\right)\mu \\
    &\hspace{1cm}  + \mathbb E [f'{}'(w^T x) ]w^T C ACw 
\end{align*}
\end{lemma}  
The second lemma allows us to integrate over $Y$, it will be proven after the proof of the proposition.
\begin{lemma}\label{lem:integration_moyenne_et_covariance_concentres}
  Given two (sequences of) random variables $\mu \in \mathbb R$ and $\nu \in \mathbb R$ such that $0<\tilde \nu \leq O(1)$ and two (sequences of) deterministic variable $\tilde \mu, \tilde \nu \in \mathbb R$ satisfying:
  \begin{align*}
    \mu \in  \tilde \mu \pm \mathcal E_2 \left(\frac{1}{\sqrt n}\right) &
    &\text{and}&
    &\nu \in  \tilde \nu \pm \mathcal E_2 \left(\frac{1}{\sqrt n}\right) ,
  \end{align*}
  if we consider a differentiable mapping $f : \mathbb R \to \mathbb R$ satisfying $\|f'\|_ \infty \leq  O(1)$, 
  then for any Gaussian random variable $z \sim \mathcal N(\mu, \nu)$, independent with $\nu$ and $\mu$:
  \begin{align*}
     \mathbb E_z[f_i(z)] \in \mathbb E[f_i(\tilde z)] \pm \mathcal E_2 \left(\frac{1}{\sqrt n}\right)
  \end{align*}
  where $\mathbb E_z$ is the expectation taken only on the variation of $z$ and $\tilde z \sim \mathcal N(\tilde \mu,\tilde \nu)$.
\end{lemma}
\begin{proof}[Proof of Proposition~\ref{pro:first_approximation_m_C_Y}]
  Thanks to Lemma~\ref{lem:Stein}, we can express for any $u \in \mathbb R^p$ symmetric, such that $\|u\| \leq O(1)$ and noting $\mathbb E_{-i} = \prod_{\genfrac{}{}{0pt}{2}{1\leq j \leq n}{j \neq i}} \mathbb E_{j}$:
\begin{align*}
  u^T m_Y
   &= \frac{1}{n}\sum_{i=1}^n \mathbb E_{-i} \left[ \mathbb E_i \left[ \xi_i(z_i) u^Tx_i \right] \right]+ O\left(\frac{\log n}{\sqrt n}\right)\\
  &= \frac{1}{n}\sum_{i=1}^n \mathbb E[\xi_i(z_i)] u^Tm_i + \mathbb E_{-i}  \left[ \mathbb E_i[\xi_i'(z_i)]u^T C_i Y_{-i} \right] + O\left(\frac{\log n}{\sqrt n}\right)\\
   &= \frac{1}{n}\sum_{i=1}^n \mathbb E[\xi_i(z_i)] u^Tm_i + \mathbb E[\xi_i'(z_i)]u^T C_i m_Y + O\left(\frac{\log n}{\sqrt n}\right)
\end{align*}
thanks to Lemma~\ref{lem:integration_moyenne_et_covariance_concentres} and Corollary~\ref{cor:statistique_Y_Y_m_i}. We can then deduce, replacing $ u $ by $u^T \tilde Q$:
\begin{align*}
  u^Tm_Y
   &= \frac{1}{n}\sum_{i=1}^n \mathbb E[\xi_i(z_i)] u^T\tilde Qm_i + O\left(\frac{\log n}{\sqrt n}\right).
\end{align*}

To estimate the covariance, we can once again deduce from Propoition~\ref{pro:lien_Y_Y_mi} that for any $ A \in \mathcal{M}_{p}$ such that $\|A\| \leq 1$:
\begin{align*}
  \tr(A \Sigma_Y) 
  &= \mathbb E[Y^TAY] = \frac{1}{n}\sum_{i=1}^n \mathbb E_{-i} \left[ \mathbb E_i \left[ \xi_i(z_i) Y^TAx_i \right] \right]+ O\left(\frac{\log n}{\sqrt n}\right)\\
  &= \frac{1}{n}\sum_{i=1}^n \mathbb E_{-i} \left[ \mathbb E_i \left[ \xi_i(z_i) Y_{-i}^TAx_i \right] \right] + \frac{1}{n}\mathbb E_{-i} \left[ \mathbb E_i \left[ \xi_i(z_i)^2 x_i^T QAx_i \right] \right]+ O\left(\frac{\log n}{\sqrt n}\right) \\
  &= \frac{1}{n}\sum_{i=1}^n  \mathbb E\left[ \xi_i(z_i)  \right] \tr (m_Y Am_i) + \mathbb E\left[ \xi_i'(z_i)  \right] \tr (\Sigma_Y A C_i) \\
  &\hspace{1cm} + \frac{1}{n} \sum_{i=1}^n \mathbb E \left[ \xi_i(z_i)^2 \tr(\Sigma_i QA) \right]+ O\left(\frac{\log n}{\sqrt n}\right)
\end{align*}
note that we could remove several terms from the formula of Lemma~\ref{lem:integration_moyenne_et_covariance_concentres}, since $\|Y_{-i}\| \leq 1$.
Replacing $A$ by $A \tilde Q$, and employing Theorem~\ref{the:concentration_resolvente_avec_diagonale_aleatoire} we then obtain:
\begin{align*}
  \tr(A \Sigma_Y) 
  = m_Y^T A m_Y + \frac{1}{n}\sum_{i=1}^n \mathbb E \left[ \xi_i(z_i)^2 \tr(\Sigma_i \tilde QA \tilde Q) \right] + O\left(\frac{\log n}{\sqrt n}\right)
\end{align*}
\end{proof}


\begin{proof}[Proof of Lemma~\ref{lem:integration_moyenne_et_covariance_concentres}]
  Let us introduce a Gaussian random variable $y \sim \mathcal N(0,1)$, independent with $\mu$ and $\nu$. We can express:
  \begin{align*}
    \mathbb E_z[f_i(z)] \in \mathbb E[f_i(\tilde z)] = \mathbb E_y[f_i(\mu + \sqrt{\nu} y)] \equiv \phi(\mu, \nu)
  \end{align*}
 The mapping $y \mapsto f_i(\mu + \sqrt{\nu} y) e^{-y^2/2}$ is bounded, we can thus differentiate $\phi$ and we can bound:
  \begin{align*}
     \frac{\partial \phi}{\partial \nu} = \mathbb E_y[\sqrt{\nu}f_i'(\mu + \sqrt{\nu} y)] \leq O(1)&
     &\frac{\partial \phi}{\partial \mu} = \mathbb E_y[f_i'(\mu + \sqrt{\nu} y)] 
     \leq O(1)
   \end{align*}
   Therefore as $O(1)$-Lipschitz transformations of $\mu,\nu$ under $\mathcal A_{\mu,\nu}$, we obtain the concentration $\phi(\mu,\nu) \in \phi(\tilde \mu, \tilde \nu) \pm \mathcal E_2(1/\sqrt n) $ (see Remark~\ref{rem:a_basse_dimension_lipschitz_egal_lineaire_equivalent_deterministe_rentre_dans_la_fonction}), which is exactly the result of the proposition.
\end{proof}
\appendix
\section{Important elements of concentration of measure Framework}\label{app:concentration_of_the_measure}
\begin{proposition}\label{pro:stabilite_lipschitz_mappings}
  In the setting of Definition~\ref{def:concentrated_sequence}, given a sequence $(\lambda_p)_{p\geq 0} \in \mathbb R_+^{\mathbb N}$, a supplementary sequence of normed vector spaces $(E_p', \Vert \cdot \Vert_p')_{p\geq 0}$ and a sequence of $\lambda_p$-Lipschitz transformations $F_p : (E_p, \Vert \cdot \Vert_p) \rightarrow (E_p', \Vert \cdot \Vert_p')$, we have
  \begin{align*}
    Z_p \propto \mathcal E_q(\sigma_p)&
    &\Longrightarrow&
    &F_p(Z_p) \propto \mathcal E_q(\lambda_p\sigma_p).
  \end{align*}
\end{proposition}

\begin{proposition}\label{pro:stabilite_concentration_lineaire_affine}
  Given two normed vector spaces $(E,\|\cdot\|_E)$ and $(F,\|\cdot\|_F)$, a random vector $Z \in E$, a deterministic vector $\tilde Z \in E$ and an affine mapping $\phi \in \mathcal A(E,F)$ such that $\|\mathcal L(\phi)\|_{\mathcal L}\leq \lambda$:
  \begin{align*}
    Z \in \tilde Z \pm \mathcal E_q(\sigma)&
    &\Longrightarrow&
    &\phi(Z) \in \phi(\tilde Z) \pm \mathcal E_q(\lambda\sigma).
  \end{align*}
\end{proposition}

The next lemma is a formal expression of the assessment that ``any deterministic vector located at a distance smaller than the observable diameter to a deterministic equivalent is also a deterministic equivalent''.
\begin{lemma}\label{lem:diametre_observable_pivot}
  Given a random vector $Z \in E$, a deterministic vector $\tilde Z \in E$ such that $Z \in \tilde Z \pm \mathcal E_q(\sigma)$, we have then the equivalence:
  \begin{align*}
    Z \in \tilde Z' \pm \mathcal E_q(\sigma)&
    &\Longleftrightarrow&
    & \left\Vert \tilde Z - \tilde Z'\right\Vert \leq O(\sigma)
  \end{align*}
\end{lemma}
\begin{remark}\label{rem:a_basse_dimension_lipschitz_egal_lineaire_equivalent_deterministe_rentre_dans_la_fonction}
  For random variables, or law rank random vectors, the notions of Lipschitz concentration and linear concentration are equivalent. More over, if $Z$ is a random variable satisfying $Z \in \mathcal E_q(\sigma)$, for any $1$-Lipschitz mapping $f: \mathbb R \to \mathbb R$, we have:
  \begin{align*}
    f(Z) \in f(\mathbb E[Z]) \pm \mathcal E_q(\sigma).
  \end{align*}
  Indeed $f(Z) \in \mathbb E[f(Z)] \pm \mathcal E_q(\sigma)$  and: $$|\mathbb E[f(Z)] -f(\mathbb E[Z]) | \leq \mathbb E [ |f(Z) - f(\mathbb E[Z])|] \leq \mathbb E[|Z - \mathbb E[Z]|] = O(\sigma),$$
  thanks to Proposition~\ref{pro:characterization_moments}. The same holds for a random vector $Z=(Z_1,\ldots, Z_d) \in \mathbb R^d$ if $d \leq O(1)$, because we can bound:
  \begin{align*}
    \mathbb E [ \|f(Z) - f(\mathbb E[Z])\|] \leq \sqrt{ \sum_{i=1}^d\mathbb E [ (f(Z_1) - f(\mathbb E[Z_d]))^2]} \leq \sqrt d O(\sigma) = O(\sigma),
  \end{align*}
  thanks again to Proposition~\ref{pro:characterization_moments}, since for all $i\in [d]$, $Z_i \in \mathcal E_q(\sigma)$
\end{remark}
Let a precise characterization of the linearly concentrated random vectors of the (sequence of) normed vector space $\mathbb R^p$ thanks to a bound on the moments, as we did in Proposition~\ref{pro:characterization_moments}. 
\begin{definition}[Moments of random vectors]\label{def:moment_vecteur_aleatoire}
  Given a random vector $X \in \mathbb R^p$ and an integer $r \in \mathbb N$, we call the ``$r^{\textit{th}}$ moment of $X$'' the symmetric $r$-linear form
$C_r^X : (\mathbb R^p)^r \to \mathbb R$ defined for any $u_1,\ldots,u_r \in \mathbb R^p$ with:
\begin{align*}
  C_r^X(u_1,\ldots,u_p) = \mathbb E \left[\prod_{i=1}^p\left(u_i^TX - \mathbb E[u_i^TX]\right)\right].
\end{align*}
When $r=2$, we retrieve the covariance matrix.
\end{definition}
Given an $r$-linear form $S$ of $\mathbb R^p$ we note its operator norm:
\begin{align*}
  \|S\| \equiv \sup_{\|u_1\| ,\ldots,\|u_r\|\leq 1}S(u_1,\ldots,u_p),
\end{align*}
when $S$ is symmetric we employ the simpler formula $\|S\| = \sup_{\|u\|\leq 1}S(u,\ldots,u)$.
We have then the following characterization that we give without proof since it is a simple consequence of the definition of linearly concentrated random vectors an Proposition~\ref{pro:characterization_moments}.
\begin{proposition}\label{pro:carcaterisation_vecteur_linearirement_concentre}
  Given $q>0$, a sequence of random vectors $X_p \in \mathbb R^p$, and a sequence of positive numbers $\sigma_p >0$, we have the following equivalence:
  \begin{align*}
    X \in \mathcal E_q(\sigma)&
    &\Longleftrightarrow&
    &\exists C,c>0, \forall p \in \mathbb N, \forall r \geq q : \|C^{X_p}_r\| \leq C \left(\frac{r}{q}\right)^{\frac{r}{q}}(c\sigma_p)^r
  \end{align*}
\end{proposition}
In particular, if we note $C = \mathbb E[XX^T] - \mathbb E[X]\mathbb E[X]^T$, the covariance of $X\in \mathcal E_q(\sigma)$, we see that $\|C\| \leq O(\sigma)$.

\subsection{Concentration of the norm}

Given a random vector $Z \in (E,\|\cdot \|)$, if $Z \in \tilde Z \pm \mathcal E_q(\sigma)$, the control on the norm $\|Z - \tilde Z \|$ can be done easily when the norm $\left\Vert  \cdot \right\Vert$ can be defined as the supremum on a set of linear forms; for instance when $(E,\|\cdot \|) = (\mathbb{R}^p \left\Vert \cdot \right\Vert_{\infty})$: $\left\Vert x \right\Vert_{\infty} =\sup_{ 1 \leq i \leq p} e_i^T x$ (where $(e_1, \ldots,e_p)$ is the canonical basis of $\mathbb R^p$). We can then bound:
\begin{align*}
    \mathbb{P}\left(\Vert Z - \tilde Z \Vert_{\infty} \geq t\right) 
    &=\mathbb{P}\left(\sup_{1\leq i \leq p} e_i^T(Z - \tilde Z) \geq t\right) \\
    &\leq \min \left(1, p \sup_{1\leq i \leq p}\mathbb{P}\left( e_i^T(Z - \tilde Z) \geq t\right)\right) \\
    &\leq \min \left(1,p C e^{-(t/c)^q}  \right) \ \ \leq \ \max(C,e) \exp \left(-\frac{t^q}{2c^q \log(p)}\right),
\end{align*}
for some $c = O(\sigma)$ and some constant $C>0$.
To manage the infinity norm, the supremum is taken on a finite set $\{e_1, \ldots e_p\}$.

Problems arise when considering the Euclidean norm satisfying for any $x \in \mathbb R^p$ the identity $\left\Vert x \right\Vert = \sup\{ u^T x, \Vert u \Vert \leq 1\}$, indeed, here the supremum is taken on the whole unit ball $\mathcal B_{\mathbb R^p} \equiv \{u \in \mathbb R^p, \Vert u \Vert \leq 1\}$ which is an infinite set.
This loss of cardinality control can be overcome if one introduces so-called $\varepsilon$-nets to discretize the ball with a net $\{u_i\}_{i \in I}$ (with $I$ finite -- $|I |<\infty$) in order to simultaneously 
\begin{enumerate}
  \item approach sufficiently the norm to ensure $$\mathbb{P}\left(\Vert Z - \tilde Z \Vert_{\infty} \geq t\right)\approx\mathbb{P}\left(\sup_{i \in I} u_i^T(Z - \tilde Z) \geq t\right),$$
  \item control the cardinality $|I|$ for the inequality $$\mathbb{P}\left(\sup_{i \in I} u_i^T(Z - \tilde Z) \geq t\right)\leq |I |\mathbb{P}\left( u_i^T(Z - \tilde Z) \geq t\right)$$ not to be too loose (see \cite{TAO12} for more detail).
\end{enumerate}
One can then show:
\begin{align}\label{eq:concentration_norme_euclidienne}
  \mathbb P(\Vert Z - \tilde Z \Vert\geq t) \leq \ \max(C,e) \exp^{-(t/c)^q/p}.
\end{align} 
The approach with $\varepsilon$-nets in $(\mathbb R^p,\|\cdot\|)$ can be generalized to any normed vector space $(E,\|\cdot \|)$ where the norm can be written as a supremum through an identity of the kind 
\begin{align}\label{eq:norm_egal_supremum}
  \forall x \in E : \|x\| = \sup_{f \in H} f(x)&
  &\text{with}\  H\subset E' \ \text{and} \ \dim(\vect(H)) < \infty,
 \end{align} for a given $H\subset E'$ and where $\vect H$ designates the subspace of $E$ generated by $H$. Such a $H\subset E'$ exists in particular when $(E,\|\cdot\|)$ is a reflexive spaces.
\begin{proposition}[\cite{JAM57}]\label{pro:charact_norm_reflexive_space}
  In a reflexive space $(E,\|\cdot\|)$:
  \begin{align*}
    \forall x \in E: \ \left\Vert x \right\Vert = \sup_{f \in \mathcal B_{E'}} f(x)&\text{ where } \mathcal B_{E'} = \{ f \in E' \ | \ \Vert f \Vert \leq 1\}.
  \end{align*}
\end{proposition}

When $(E,\|\cdot\|)$ has an infinite dimension and is not reflexive, it is sometimes possible to establish \eqref{eq:norm_egal_supremum} for some $H \subset E$ in some cases (most of them appearing when $\|\cdot\|$ is a semi-norm).
Without going deeper into details, we introduce the notion of \textit{norm degree} that will help us adapting to other normed vector space the concentration rate $p$ appearing in the exponential term of concentration inequality \eqref{eq:concentration_norme_euclidienne} (concerning $(\mathbb R^p, \| \cdot\|)$).
\begin{definition}[Norm degree]\label{def:norm_degree}
  Given a normed (or seminormed) vector space $(E, \Vert \cdot \Vert)$, and a subset $H\subset E'$, the degree $\eta_H$ of $H$ is defined as~:
  \begin{itemize}
      \item $\eta_H \equiv \log(| H|)$ if $H$ is finite,
      \item $\eta_H \equiv \dimm(\vect H)$ if $H$ is infinite.
  \end{itemize}
  If there exists a subset $H \subset E'$ such that \eqref{eq:norm_egal_supremum} is satisfied, we then denote $\eta(E, \Vert \cdot \Vert)$, or more simply $\eta_{\Vert \cdot \Vert}$, the degree of $\Vert \cdot \Vert$, defined as~:
  \begin{align*}
      \eta_{\Vert \cdot \Vert}=\eta(E, \Vert \cdot \Vert)\equiv\inf \left\{\eta_H, H\subset E' \ | \ \forall x \in E, \Vert x \Vert = \sup_{f \in H}f(x)\right\}.
  \end{align*}
\end{definition}

\begin{example}\label{exe:norm_degree}
    We can give some examples of norm degrees~:
    \begin{itemize}
        \item $\eta \left( \mathbb R^p, \Vert \cdot \Vert_\infty \right) = \log(p)$ ($H = \{x \mapsto e_i^Tx, 1\leq i\leq p\}$),
        \item $\eta \left( \mathbb R^p, \Vert \cdot \Vert \right) = p$ ($H = \{x \mapsto u^Tx, u \in \mathcal B_{\mathbb R^p}\}$),
        \item $\eta \left( \mathcal M_{p,n}, \Vert \cdot \Vert \right) = n+p$ ($H = \{M \mapsto u^TMv, (u,v)\in \mathcal B_{\mathbb R^p} \times \mathcal B_{\mathbb R^n}\}$),
        \item $\eta \left( \mathcal M_{p,n}, \Vert \cdot \Vert_F \right) = np$ ($H = \{M \mapsto \tr(AM), A \in \mathcal M_{n,p}, \|A\|_F\leq 1\}$),
        \item $\eta \left( \mathcal M_{p,n}, \Vert \cdot \Vert_* \right) = np$ ($H = \{M \mapsto \tr(AM), A \in \mathcal M_{n,p}, \|A\|\leq 1\}$)\footnote{$\|\cdot\|_*$ is the nuclear norm defined for any $M \in \mathcal M_{p,n}$ as $\|M\|_* = \tr(\sqrt{MM^T})$ it is the dual norm of $\| \cdot \|$, which means that for any $A,B \in \mathcal M_{p,n}$, $\tr(AB^T) \leq \|A\| \|B\|_*$.}.
    \end{itemize}
\end{example}
Depending on the vector space we are working in, we can then employ those different examples and the following proposition to set the concentration of the norm of a random vector.
\begin{proposition}\label{pro:tao_conc_exp}
Given a reflexive vector space $(E, \Vert \cdot \Vert)$ and a concentrated vector $Z \in E$ satisfying $Z \in \tilde Z \pm \mathcal E_q(\sigma)$:
  \begin{align*}
       &\Vert Z - \tilde Z \Vert \in O \left(\eta_{\Vert \cdot \Vert}^{1/q}\sigma\right) \pm \mathcal E_{q}\left(\eta_{\Vert \cdot \Vert}^{1/q}\sigma\right).
  \end{align*}
  
\end{proposition}
\begin{remark}\label{rem:concentration_norme_lineaire_lipschitz}
  When $Z \propto \mathcal E_q(\sigma)$, we have of course the better concentration $\Vert Z - \tilde Z \Vert \propto \mathcal E_{q}\left(\sigma\right)$ but the bound $\mathbb E \left[\Vert Z - \tilde Z \Vert\right]\leq O \left(\eta_{\Vert \cdot \Vert}^{1/q}\sigma\right)$ can not be improved.
\end{remark}

\begin{example}\label{exe:borne_esp_norm_vecteur_lin_conc}
Given two random vectors $Z \in \mathbb{R}^p$ and $M \in \mathcal M_{p,n}$: 
\begin{itemize}
  \item if $Z\propto \mathcal E_2$ in $(\mathbb{R}^p,\left\Vert \cdot\right\Vert)$~: $\mathbb{E}\left\Vert Z\right\Vert\leq\Vert \mathbb E[Z] \Vert + O(\sqrt{p})$,
  \item if $M \propto \mathcal E_2$ in $(\mathcal M_{p,n}, \left\Vert \, \cdot \,\right\Vert)$~: 
  $\mathbb{E}\left\Vert M\right\Vert \leq \Vert \mathbb E[M]\Vert + O(\sqrt{p+n}),$
  \item if $M \propto \mathcal E_2$ in $(\mathcal M_{p,n}, \left\Vert \, \cdot \,\right\Vert_F)$~:
  $\mathbb{E}\left\Vert M\right\Vert \leq \Vert \mathbb E[M]\Vert_F + O(\sqrt{pn})$.
  \item if $M \propto \mathcal E_2$ in $(\mathcal M_{p,n}, \left\Vert \, \cdot \,\right\Vert_*)$~:
  $\mathbb{E}\left\Vert M\right\Vert_* \leq \Vert \mathbb E[M]\Vert_* + O(\sqrt{pn})$.
\end{itemize}
\end{example}

\subsection{Concentration of basic operations}
Returning to Lipschitz concentration, if we want to control the concentration of the sum $X+Y$ or the product $XY$ of two random vectors $X$ and $Y$, we first need to express the concentration of the concatenation $(X,Y)$. This last result is very easy to obtain in the class of linearly concentrated random vector since it is a consequence of Proposition~\ref{pro:concentration_serie_variables_concentres}, in the class of Lipschitz concentrated vectors, the concentration of $(X,Y)$ is far more complicated, and independence here plays a central role (unlike for linear concentration). 

To understand the issue, let us give an example where $X$ and $Y$ are concentrated but not $(X,Y)$. Consider $X$, uniformly distributed on the sphere $\sqrt p\mathbb S^{p-1}$ and $Y=f(X)$ where for any $x =(x_1,\ldots,x_p) \in \mathbb R^p$, $f(x) = x$ if $x_1\geq 0$ and $f(x) = -x$ otherwise. We know that all the linear observations of $X+Y$ are concentrated thanks to Proposition~\ref{pro:concentration_serie_variables_concentres}, but it is not the case for all the Lipschitz observations. Indeed, it is straight forward to see that the diameter of the observation $\|X+Y\|$ (which is a $\sqrt 2$-Lipschitz transformation of $(X,Y)$) is of order $O(\sqrt p)$ like the metric diameter of $X+Y$ however it should be of order $O(1)$ if $(X,Y)$ would be concentrated. This effect is due to the fact that the mapping $f$ is clearly not Lipschitz, and $Y$ in a sense ``defies'' $X$.

Still there exists two simple ways to obtain the concentration of $(X,Y)$, the first one being deduced from any identity $(X,Y) = \phi(Z)$ with $Z$ concentrated and $\phi$ Lipschitz.
And the second ones arises from independence relations as seen in next lemma.
\begin{lemma}\label{lem:concentration_(X,Y)_independent}
Given $(E,\Vert \cdot \Vert)$, a sequence of normed vector spaces and two sequences of independent random vectors $X,Y \in E$, if we suppose that $X \propto  \mathcal E_{q}(\sigma)$ and $Y \propto \mathcal E_{r}(\rho)$ (where $q,r >0$ are two positive constants and $\sigma,\rho \in \mathbb R_+^{\mathbb N}$ are two sequences of positive reals):
  \begin{align*}
    (X,Y) \propto \mathcal E_{q}\left( \sigma \right) + \mathcal E_{r}\left( \rho \right)&
    &\text{in} \ \ (E^2, \|\cdot\|_{\ell^\infty}),
  \end{align*}
  where we note for all $x,y \in E^2$, $\|(x,y)\|_{\ell^\infty} = \max(\|x\|,\|y\|)$\footnote{One could have also considered a big number of equivalent norms like $\|(x,y)\|_{\ell^1} = \|x\| + \|y\|$ or $\|(x,y)\|_{\ell^2} = \sqrt{\|x\|^2  + \|y\|^2}$.}.
  Following our formalism, this means that there exist two positive constants $C,c>0$ such that $\forall n\in \mathbb N$ and for any $1$-Lipschitz function $f :(E_n^2, \|\cdot \|_{\ell^\infty}) \rightarrow (\mathbb R, | \cdot |)$, $\exists d_n = O(\rho_n), \forall t>0:$ 
     \begin{align*}
      \mathbb P \left(\left\vert f(X_n, Y_n) - f(X_n', Y_n')\right\vert\geq t\right) \leq C e^{(t/c\sigma_n)^q} + C e^{(t/cd_n)^r}.
     \end{align*} 
\end{lemma}

The sum being a $2$-Lipschitz operation (for the norm $\|\cdot \|_{\ell^\infty}$), the concentration of $X+Y$ is easy to handle and directly follows from Lemma~\ref{lem:concentration_(X,Y)_independent}. For the product of random variable, on can employ:
\begin{proposition}[\cite{}]\label{pro:concentration_produit_variable}
  Given two (sequences of) random variables $Z, Z'\in \mathbb R$ and four sequence of parameters $a,a' \in \mathbb R$ and $\sigma,\sigma' >0$, we have the implication\footnote{Notation $Z \in a \pm\mathcal E_{q_1}(\sigma_1) + E_{q_2}(\sigma_2)$ (with $Z$ being a random variable) signifies that there exist two constants $C,c>0$ such that for all $n \in \mathbb N$:
  \begin{align*}
    \mathbb P \left( |Z-a|\geq t \right)
    \leq C e^{-(ct/\sigma_1)^{q_1}} +  C e^{-(ct/\sigma_2)^{q_2}}.
  \end{align*}}:
  \begin{align*}
    \left\{\begin{aligned}
    &Z \in a \pm E_q(\sigma)\\
    & Z' \in a' \pm E_q(\sigma')
    \end{aligned}\right.&
    &\Longrightarrow&
    &ZZ' \in aa' \pm E_q(\sigma a' + \sigma' a) + E_{q/2}(\sigma \sigma').
  \end{align*}
\end{proposition}

The product of random vectors is harder to treat but satisfies similar concentration inequality. We advice the reader to have a look at the work already done in \cite{LOU21HV} where we presented a result giving the concentration of a product of $m$ vectors $Z_1,\ldots, Z_m$. From an hypothesis $(Z_1,\ldots, Z_m) \propto \mathcal E_2$, one can deduce the concentration
\begin{align}\label{eq:concentration_produit_m_termes}
   Z_1 \odot \cdots \odot Z_m \propto \sum_{l=1}^m\mathcal E_{2/l}(\sigma_l),
\end{align} 
 where each $\sigma_l$ depends on the expectations of the norms $(\mathbb E[\| Z_i\|_i])_{i \in [m]}$ for a good choice of norms.

\section{Definition of the expectation}\label{app:definition_expectation}
Given a concentrated random vector $Z \in (E,\|\cdot\|)$, we will need at one point (Theorems~\ref{the:COncentration_lineaire_solution_concentrated_equation_lipschitz} and~\ref{the:lipschitz_COncentration_solution_conentrated_equation_Phi_concentre_norme_infinie}) to be able to consider its expectation. Thanks to the characterization of exponential concentration with a bound on moments, we already know that, if $Z \propto \mathcal E_q(\sigma)$, then for any Lipschitz mapping $f : \mathbb E \mapsto \mathbb R$, the functional $f(Z)$ admits an expectation $\mathbb E[f(Z)] = \int_E f(z) d \mathbb P(Z=z)$. This definition can be first generalized when $E$ is a \textit{reflexive space}. 

Given a normed vector space $(E,\|\cdot \|)$, we denote $(E', \|\cdot \|)$ the so-called ``strong dual'' of $E$, composed of the continuous linear forms of $E$ for the norm $\|\cdot\|$. The norm $\|\cdot \|$ (written the same way as the norm on $E$ for simplicity -- no ambiguity being possible) is called the \textit{strong norm} of $E'$ and defined as follows.
\begin{definition}\label{def:strong_norm}
  Given a normed vector space $(E,\|\cdot\|)$, the strong norm $\|\cdot \|$ is defined on $E'$ as:
  \begin{align*}
    \forall f \in E', \|f \| = \sup_{\|x \| \leq 1} |f(x)|.
  \end{align*}
\end{definition}
To be able to define an expectation in $E$ we first assume that $E$ is \textit{reflexive}.
To this end, we need to first define a ``topological bidual'' of $E$, denoted $(E'{}', \| \cdot \|)$ and defined by $E'{}' = (E')'$ with norm the strong norm of the dual of $E'$.
\begin{definition}\label{def:natural_embedding}
  The ``natural embedding'' of $E$ into $E'{}'$ is defined as the mapping:
  \begin{align*}
    J :
    \begin{aligned}[t]
       E &&\longrightarrow &&E'{}'\hspace{0.6cm}\\
       x &&\longmapsto &&(E' \ni f \mapsto f(x)).
    \end{aligned}
   \end{align*} 
  It can be shown that $J$ is always one-to-one, but not always onto; however, when $J$ is bijective, we say that $E$ is reflexive.
 \end{definition} 

If $E$ is reflexive, then it can be identified with $E'{}'$ (this is in particular the case of any vector space of finite dimension but also of any Hilbert space). One can then define the expectation of any concentrated vector $X \propto \mathcal E_q(\sigma)$ as follows:
\begin{definition}\label{def:expectation_reflexive_space}
  Given a random vector $Z$ of a reflexive space $(E,\|\cdot \|)$, if the mapping $E' \ni f \mapsto \mathbb E[f(Z)] \in \mathbb R$ is continuous on $E'$, we define the expectation of $Z$ as the vector:
  \begin{align}\label{eq:def1_expectation}
    \mathbb E[Z] = J^{-1}(f \mapsto \mathbb E[f(Z)]).
  \end{align}
\end{definition}

\begin{remark}\label{rem:reflexif_complet}
  A reflexive space is a complete space (since it is in bijection with a dual space). It satisfies in particular the Picard Theorem which states that any contractive mapping $f:E\to E$ ($\forall x,y \in E, \|f(x) - f(y) \leq (1-\varepsilon) \|x - y\|$ with $\varepsilon >0$) admits a unique fixed point $y = f(y)$. This property will be particularly interesting when considering the concentration of fixed point of concentrated functions of Reflexive spaces in Section~\ref{sec:concentration_of_the_solutions_to_conc_eq}.
\end{remark}
\begin{lemma}\label{lem:fE_egal_Ef}
  Given a reflexive space $(E,\|\cdot \|)$, a random vector $Z\in E$ and a continuous linear form $f\in E'$: $$f(\mathbb E[Z])  = \mathbb E[f(Z)].$$
\end{lemma}
\begin{proof}
  It is just a consequence of the identity:
  \begin{align*}
    f(\mathbb E[Z]) = J(\mathbb E[Z])(f) = J \left(J^{-1}(f \mapsto \mathbb E[f(Z)])\right)(f) = \mathbb E[f(Z)].
  \end{align*}
\end{proof}
\begin{proposition}\label{pro:existence_expectation_reflexive_concentrated_vector}
  Given a reflexive space $(E,\|\cdot \|)$ and a random vector $Z \in E$, if $Z \propto \mathcal E_q(\sigma)$, then $\mathbb E[Z]$ can be defined with Definition~\ref{def:expectation_reflexive_space}.
\end{proposition}
\begin{proof}
  We just need to show that $f \mapsto \mathbb E[f(Z)]$ is continuous. There exists $K_p > 0$ such that $\mathbb P(\|Z_p \| \leq K_p ) \geq \frac{1}{2}$, so that for any $f \in E'$, $\mathbb P(f(Z_p) \leq K_p \|f\|) \geq \frac{1}{2} $. Therefore, by definition, for any median $m_f$ of $f(Z_p)$, $m_f \leq K_p \|f\|$. The two concentration $f(Z) \in \mathbb E[f(Z)] \pm \mathcal E_q(\sigma)$ and $f(Z) \in m_f \pm \mathcal E_q(\sigma)$ then allows us to obtain a similar bound on $|\mathbb E[f(Z)]|$ which allows us to state that 
  the mapping $f \mapsto |\mathbb E[f(Z)]|$ is continuous.
\end{proof}

It is still possible to define a notion of expectation when $E$ is not reflexive but is a functional vector space having value on a reflexive space $(F,\|\cdot\|)$; for instance a subspace of $F^G$, for a given set $G$.
\begin{definition}\label{def:expectation_functional_spaces}
  Given reflexive space $(F,\|\cdot\|)$, a given set $G$, a subspace $E \subset F^G$, and a random vector $\phi \in E$, if for any $x \in F$, the mapping $F' \ni f \mapsto \mathbb E[f(\phi(x))]$ is continuous, we can defined the expectation of $\phi$ as:
  \begin{align*}
    \mathbb E[\phi] : x \mapsto \mathbb E[\phi(x)].
  \end{align*}
\end{definition}
\begin{remark}\label{rem:definition_expectation_2_consistance}
  When the space $E\subset F^G$ is endowed with a norm $\|\cdot\|$ such that $(E,\|\cdot \|)$ is reflexive and $\forall x \in G$, $E \ni \phi \mapsto \phi(x) $ is continuous, then there is no ambiguity on the definitions. Indeed, if we note $\mathbb E_1[\phi]$ and $\mathbb E_2[\phi]$, respectively the expectation of $\phi$ given by Definition~\ref{def:expectation_reflexive_space} and Definition~\ref{def:expectation_functional_spaces}, we can show for any $x \in F$ and any $f\in F'$:
  \begin{align*}
    f(\mathbb E_1[\phi](x))
    &= \tilde f(\mathbb E_1[\phi]) = \mathbb E[\tilde f(\phi)] = \mathbb E[ f(\phi(x))] = f(\mathbb E[ \phi(x)]) = f(\mathbb E_2[\phi](x)),  
  \end{align*}
  where $\tilde f : E \ni\psi \rightarrow f(\psi(x))$ is a continuous linear form. Since this identity is true for any $f \in E'$, we know by reflexivity of $E$ that $\forall x \in E$: $\mathbb E_1[\phi](x) = \mathbb E_2[\phi](x)$. This directly implies that $\mathbb E_1[\phi] = \mathbb E_2[\phi]$.
\end{remark}

\begin{remark}\label{rem:expectation_of_concentrated_random_mappings}
  Given a random mapping $\phi \in E \subset F^G$ with $(F, \|\cdot \|)$ reflexive, we can then deduce from Proposition~\ref{pro:existence_expectation_reflexive_concentrated_vector} that if for all $x \in G$, $\phi(x) \propto \mathcal E_q(\sigma)$ in $(F, \| \cdot \|)$ then we can define $\mathbb E[\phi]$. 
  With a different formalism, we can endow $F^G$ with the family of semi-norms $(\|\cdot \|_x)_{x \in G}$ defined for any $f \in F^G$ as $\|f\|_x = \|f(x)\|$; then if for all $x \in G$, $\phi\propto \mathcal E_q(\sigma)$ in $(F^G, \|\cdot \|_x)$, it is straightforward to set that $\mathbb E[\phi]$ is well defined.
  This is in particular the case if $E = \mathcal B(G,F)$ is the set of bounded mappings from $G$ to $F$ and $\phi \propto \mathcal E_q(\sigma)$ in $(F^G,\|\cdot \|_{\infty})$ where for all $f \in F^G$, $\|f\|_\infty = \sup_{x \in G} \| f(x)\|$. 
\end{remark}
\section{Concentration of $Y$ solution to $Y = \phi(Y)$ when $\phi$ is locally Lipschitz}\label{app:localement_lipschitz}
For that we first need a preliminary lemma that will allow us to set that a mapping contracting in a sufficiently large compact admits a fixed point. It expresses through the introduction of a new semi-norm defined for any $f \in \mathcal F(E)$, locally Lipschitz as:
\begin{align*}
  \|f\|_{\mathcal L,\mathcal B(y_0, r)} = \sup_{x,z \in  \mathcal B(y_0, r)} \frac{\|f(x) - f(z)\|}{\|x-z\|}.
\end{align*}
\begin{lemma}\label{lem:boule_de_point_fixe}
  Given a mapping $\phi \in \mathcal F(E)$, if there exist two constants $\delta, \varepsilon>0$ and a vector $y_0 \in E$ such that:
  \begin{align*}
    \|\phi \|_{\mathcal L, \mathcal B(y_0,\delta)} \leq 1-\varepsilon&
    &\text{and}&
    &\|\phi(y_0) - y_0\| \leq \delta \varepsilon,
   \end{align*} 
  then for any $y \in \mathcal B(y_0,\delta\varepsilon)$ and any $k\in \mathbb N$:
  \begin{align}\label{eq:borne_fk_y0_K}
    \left\Vert \phi^k(y) - y_0\right\Vert  \leq \delta.
  \end{align}
\end{lemma}
\begin{proof}
  That can be done iteratively. For $k=0$, it is obvious since $\|y - y_0\| \leq \varepsilon\delta \leq \delta$ ($\varepsilon <1$). Now if we suppose that \eqref{eq:borne_fk_y0_K} is true for all $l<k$ and $y \in \mathcal B(y_0,\varepsilon\delta)$ (thus in particular for $y = y_0$), we can bound (under $\mathcal A_{\phi^\infty}$):    
  \begin{align*}
  \left\Vert \phi^k(y) - y_0\right\Vert
    &\leq \left\Vert \phi^k(y) - \phi^k(y_0)\right\Vert + \left\Vert \phi^k(y_0) - y_0\right\Vert \nonumber\\
    &\leq (1-\varepsilon)^k \left\Vert y - y_0\right\Vert + \sum_{i=1}^k \left\Vert \phi^i(y_0) - \phi^{i-1}(y_0)\right\Vert \nonumber\\
    &\leq (1-\varepsilon)^k \varepsilon \delta + \sum_{i=1}^k (1-\varepsilon)^{i-1}\left\Vert \phi(y_0) - y_0\right\Vert\nonumber\ \ \leq \  \delta.
  \end{align*}
\end{proof}
\begin{remark}\label{rem:m_peu-tendre_vers_linfini}
  Note that in the previous proof $k$ can tend to $\infty$ without any change in the concentration bounds. This is due to the fact that for any $l \in [k]$ we used the large bounds:
  \begin{align*}
    (1-\varepsilon)^l \leq 1&
    &\text{and}&
    &\sum_{i=1}^l (1-\varepsilon)^i \leq \frac{1}{\varepsilon}.
  \end{align*}
\end{remark}

\begin{theorem}\label{the:COncentration_lineaire_solution_concentrated_equation_lipschitz}
  Let us consider a (sequence of) reflexive vector space $(E, \| \cdot \|)$ admitting a finite norm degree that we note $\eta$. Given $\phi\in \lip(E)$, a (sequence of) random mapping, we suppose that there exists a (sequence of) integer $\sigma >0$, a constant $\varepsilon>0$ such that for any constant $K>0$, there exists a (sequence of) highly probable event $\mathcal A_K$ such that:
  \begin{itemize}
    \item there exists a (sequence of) deterministic vector $y_0 \in E$ satisfying:
  $$y_0 = \mathbb E_{\mathcal A_{K}}[\phi(y_0)]$$
    \item for all $y \in \mathcal B(y_0, K\sigma \eta^{1/q})$ and for all (sequence of) integer $k$ such that $k \leq O(\log(\eta))$,
  \begin{align*}
    \phi^k(y)\overset{\mathcal A_K}{\in} \mathcal E_q \left(\sigma\right) \ | \ e^{-\eta} \ \text{in} \ (E, \| \cdot \|).
  \end{align*} 
    \item $\mathcal A_K \subset \{\left\Vert \phi\right\Vert_{\mathcal L,\mathcal B(y_0, K\sigma \eta^{1/q})} \leq 1-\varepsilon\}$,
   \end{itemize} 
  then with high probability (bigger than $1-Ce^{-c\eta}$ for some constants $C,c>0$), the random equation
  \begin{align*}
    Y = \phi(Y)
  \end{align*}
  admits a unique solution $Y\in E$ satisfying the linear concentration:
  \begin{align*}
    Y \in \mathcal E_q \left(\sigma\right) \ | \ e^{-\eta}.
  \end{align*}
\end{theorem}

\begin{proof}
  We know that, $(\phi(y_0) \ | \ \mathcal A_1) \in y_0 \pm \mathcal E_q(\sigma)$, so in particular Proposition~\ref{pro:tao_conc_exp} implies that there exist three constants $C,c,K>0$ such that:
  \begin{align*}
    \mathbb P \left(\left\Vert \phi(y_0) - y_0\right\Vert\geq K\sigma\eta^{1/q}\varepsilon \ | \ \mathcal A_1\right) \leq Ce^{-\eta/c},
  \end{align*}
  Let us then note $\mathcal A_\nu = \mathcal A_1 \cap  \mathcal A_{K} \cap \{\left\Vert \phi(y_0) - y_0\right\Vert\geq K\sigma\eta^{1/q}\varepsilon\}$, we can bound $\mathbb P(\mathcal A_\nu) \leq C'e^{-\eta/c'}$ for some constants $C',c'$, and we know from Lemma~\ref{lem:boule_de_point_fixe} and our hypotheses that $\forall k\in \mathbb N$:
  \begin{align*}
    \phi^k(y_0) \in \mathcal B(y_0,K\sigma\eta^{1/q}),
  \end{align*}
  (since $y_0 \in \mathcal B(y_0,K\sigma\eta^{1/q}\varepsilon)$).
  Therefore since $\mathcal A_{Y} \subset\mathcal A_{K} \subset \{\|\phi\|_{\mathcal L, \mathcal B(y_0,K\sigma\eta^{1/q})} \leq 1-\varepsilon\}$, the sequence $(\phi^k(y_0))_{k\in \mathbb N}$ is a Cauchy sequence and it converges to a random vector $Y \in E$ satisfying $Y = \phi(Y)$ ($E$ is complete because it is reflexive).

  We now want to show that $Y$ is concentrated. Following the steps of the proof of Proposition~\ref{pro:COncentration_solution_conentrated_equation_phi_affine_k_leq_log_eta}, one sees that it is sufficient to show that $\left\Vert \tilde Y - y_0\right\Vert = O(\sigma \eta^{1/q}  )$ for $\tilde Y$ defined as the unique solution to the equation $\tilde Y = \mathbb E_{\mathcal A_\nu}[\phi^k(\tilde Y)]$ already introduced in the proof of Proposition~\ref{pro:COncentration_solution_conentrated_equation_phi_affine_k_leq_log_eta} and $k = \lceil - \frac{\log(\eta)}{q \log(1-2\varepsilon)} \rceil$. Let us bound:
  \begin{align*}
    \left\Vert \tilde Y - y_0\right\Vert
    &\leq \left\Vert \mathbb E_{\mathcal A_\nu} \left[\phi^k(\tilde Y) - \phi^k(y_0)\right]\right\Vert + \left\Vert \mathbb E_{\mathcal A_\nu}[\phi^k(y_0) - y_0 ]\right\Vert\\
    &\leq \mathbb E_{\mathcal A_\nu} \left[\left\Vert \phi^k(\tilde Y) - \phi^k(y_0)\right\Vert  \right] + \mathbb E_{\mathcal A_\nu} \left[\left\Vert \phi^k(y_0) - y_0\right\Vert \right]\\
    &\leq \mathbb E_{\mathcal A_\nu} \left[\|\phi \|^k \right] \left\Vert \tilde Y - y_0\right\Vert + \mathbb E_{\mathcal A_\nu} \left[\sum_{i=1}^k \|\phi \|_{\mathcal L}^i\left\Vert \phi(y_0) - y_0 ]\right\Vert \right]\\
    &\leq (1-\varepsilon) \left\Vert \tilde Y - y_0\right\Vert + K\sigma \eta^{1/q}.
  \end{align*}
  Thus $\left\Vert \tilde Y - y_0\right\Vert \leq \frac{K\sigma \eta^{1/q}}{\varepsilon}  $, so that, noting $\mathcal A_\nu' \equiv \mathcal A_{\frac{K\sigma \eta^{1/q}}{\varepsilon}} \cap \mathcal A_\nu$, by hypothesis:
  $$\phi^k(\tilde Y) \overset{\mathcal A_\nu'}\propto \mathcal E_q \left(\sigma\right) \ | \ e^{-\theta},$$
  the rest of the proof is then exactly the same as the proof of Proposition~\ref{pro:COncentration_solution_conentrated_equation_phi_affine_k_leq_log_eta}.
\end{proof}
Let us now give a more flexible result involving the norms $\|\cdot\|_{\mathcal B(y_0, r)}$ for $r>0$ but also the semi-norms $\|\cdot\|_{\mathcal L,\mathcal B(y_0, r)}$.
The following Lemma might seem a bit complicated and artificial, it is however perfectly adapted to the requirement of Theorem~\ref{the:lipschitz_COncentration_solution_conentrated_equation_Phi_concentre_norme_infinie} generalizing Theorem~\ref{the:COncentration_lineaire_solution_concentrated_equation_lipschitz} to the case of Lipschitz concentrated mappings $\phi$ locally Lipschitz. 
\begin{lemma}\label{lem:concentration_composition_intelligente_2}
  Given a normed vector space $(E,\|\cdot \|)$ whose norm degree is noted $\eta$, a vector $y_0 \in E$ and a (sequence of) random mapping $\phi\in \mathcal F^\infty(E)$, let us suppose that there exists  a (sequence of) constant $\varepsilon>0$ ($\varepsilon\geq O(1)$) such that for any constant $K>0$, there exists  a (sequence of) event $\mathcal A_K$:
   $$\phi \overset{\mathcal A_K}\propto \mathcal E_q \left(\sigma\right) \ | \ e^{-\eta} \ \ \text{in  } \ \ \left( \mathcal F(E), \| \cdot \|_{\mathcal B(y_0,K\sigma \eta^{1/q})} \right),$$ 
  and there exist two constants $C,c$ such that 
  \begin{align*}
    \mathbb P \left(\|\phi \|_{\mathcal L, \mathcal B(y_0, K\sigma \eta^{1/q})} \geq 1- \varepsilon \ | \ \mathcal A_K\right) \leq C e^{-c\eta},
  \end{align*}
  then for any constant $K'$ ($K' \geq O(1)$) we have:
  $$\phi^m \propto \mathcal E_q\left(\sigma \right) \ | \ e^{-\theta'} \ \ \text{in} \ \ ( \mathcal F^\infty(E), \| \cdot \|_{\mathcal B(y_0,K'\sigma \eta^{1/q})}),$$
\end{lemma}
\begin{proof}
  As in the proof of Theorem~\ref{the:COncentration_lineaire_solution_concentrated_equation_lipschitz}, let us introduce three constants $C,c>0$ and $K >K'/\varepsilon$ such that:
  \begin{align*}
    \mathbb P \left(\left\Vert \Phi(y_0) - y_0\right\Vert\geq K \sigma \eta^{1/q} \varepsilon  \ | \ \mathcal A_1\right) \leq Ce^{-\eta/c},
  \end{align*}  
  Let us set: 
  $$\mathcal A_{\phi^\infty} \equiv \mathcal A_{K\varepsilon} \cap \mathcal A_1 \cap \left\{\|\phi \|_{\mathcal L, \mathcal B_E(y_0, K\sigma \eta^{1/q})} \leq 1- \varepsilon\right\},$$
  We know from Lemma~\ref{lem:boule_de_point_fixe} that 
  for all $k\in \mathbb N$, and for all $f \in \mathcal A_{\phi^ \infty}$ (recall that we identify $\mathcal A_{\phi^ \infty}$ with $\phi(\mathcal A_{\phi^ \infty})$):
  \begin{align*}
    f^k(y) \in \mathcal B(y_0, K),
  \end{align*}
  since $\|y - y_0\| \leq K'\sigma \eta^{1/q} \leq K\sigma \eta^{1/q} \varepsilon$.
  For any $f,g \in \mathcal A_{\phi^\infty}$:
  \begin{align*}
    \|f^m - g^m\|_{\mathcal B(y_0,K'\sigma \eta^{1/q})}
    &\leq \sum_{i=1}^m \| f \|_{\mathcal L, \mathcal B(y_0,K\sigma \eta^{1/q})}^{i-1} \|f(g^{m-i}(y)) - g(g^{m-i}(y)) \| \\
    & \leq \frac{1}{\varepsilon} \|f - g \|_{\mathcal B(y_0,K\varepsilon\sigma \eta^{1/q})} .
   \end{align*}
   Thus the mapping $f \mapsto f^m$ is $\frac{1}{\varepsilon}$-Lipschitz on $(\mathcal A_{\phi^\infty}, \| \cdot \|_\mathcal B(y_0,K'\sigma \eta^{1/q}))$, which directly implies the result of the lemma.
\end{proof}

The next theorem is an important improvement (allowed by Lipschitz concentration) of Proposition~\ref{pro:lipschitz_COncentration_solution_conentrated_equation_hypothese_Psi_k_concentre} in that it only take as hypothesis the concentration of $\phi$ (and not of all its iterates). It also has the advantage to assume only a local control on the Lipschitz character of $\phi$ to describe a larger range of settings.
\begin{proposition}\label{pro:lipschitz_COncentration_solution_conentrated_equation_phi_concentre_norme_infinie_localement_lipschitz}

  Let us consider a (sequence of) reflexive vector space $(E, \| \cdot \|)$ we suppose that $\|\cdot \|$ norm degree is noted $\eta$ (of course $\eta \geq O(1)$), a (sequence of) random mapping $\phi\in \lip(E)$, a given constant $\varepsilon>0$ ($\varepsilon\geq O(1)$) and $\mathcal A_\phi$, a (sequence of) highly probable events such that $\mathbb P \left(\mathcal A_\phi\right) \leq C e^{-c\eta}$ for two constants $C,c>0$ and $\mathcal A_\phi\subset \{\|\phi\|_{\mathcal L} \leq 1-\varepsilon\}$. Introducing $y_0$, a (sequence of) deterministic vector $y_0 \in E$ satisfying:
  $$y_0 = \mathbb E[\phi(y_0)],$$
  if we suppose that for any constant $K>0$: 
  $$\phi \propto \mathcal E_q(\sigma) \ | \ e^{-\eta} \ \ \text{in  } \ \ ( \mathcal F(E), \| \cdot \|_{\mathcal B(y_0,K\sigma \eta^{1/q})}),$$ 
  then there exists an highly probable event $\mathcal A_\nu$, such that, under $\mathcal A_\nu$, the random equation
  \begin{align*}
    Y = \phi(Y)
  \end{align*}
  admits a unique solution $Y\in E$ satisfying the Lipschitz concentration:
  \begin{align*}
    Y \propto \mathcal E_q \left(\sigma\right) \ | \ e^{-\eta}.
  \end{align*}
  
\end{proposition}

\bibliographystyle{elsarticle-harv}
\bibliography{C:/Users/cosme/Documents/Travail/These/Biblio/biblio}

\end{document}